\documentclass{article}
\usepackage[a4paper]{geometry}
\usepackage[pdftex]{graphicx}
\usepackage{amsmath}
\usepackage{amssymb}
\usepackage{amsthm}
\usepackage{empheq}
\usepackage{array}
\usepackage{longtable}
\usepackage{adjustbox}
\usepackage{multirow}
\usepackage{float}
\usepackage{xcolor}
\usepackage{tikz}
\usetikzlibrary{shapes,patterns}
\usetikzlibrary{topaths,calc}
\usetikzlibrary{decorations.pathmorphing}
\newtheorem{thm}{Theorem}

\newtheorem{conjecture}{Conjecture}

\newtheorem{lemme}{Lemma}
\usepackage{pgfplots}
\pgfplotsset{compat=1.18}
\usepackage[ruled,vlined]{algorithm2e}
\usepackage[square,numbers]{natbib}
\usepackage{url}
\usepackage{doi}
\usepackage{appendix}
\usepackage{comment}
\title{\textbf{Constructing Two Completely Independent Spanning Trees in Dual-Cubes}}
\author{Mohammed Lalou, Nader Mbarek, Abdallah Skender, Olivier Togni\\
        Université Bourgogne Europe, LIB UR 7534, Dijon}
\begin{document}
\maketitle
\begin{abstract}
In this paper, we prove the existence of two completely independent spanning trees in the $n$-dimensional dual-cube $F_n$, a variant of the hypercube, for every $n \geq 5$. To this end, we use the hypercube structure of the clusters of $F_n$ to extend the construction of CIST from the $(n-1)$-dimensional hypercube to the dual-cube. In addition, we provide a recursive algorithm that builds the two trees while improving their diameters. Finally, we propose a conjecture regarding the existence of $k$ completely independent spanning trees in the dual-cube.

\end{abstract}

\section{Introduction}
\label{sec1}

The ability to maintain communication in the presence of failures is essential in computer networks. Such a network can be modeled as a graph in which vertices represent devices and edges represent communication links. A spanning tree can then provide a logical structure for the transmission of information. However, relying on a single tree can make the network vulnerable, as the failure of a link or a device can disrupt communication. A natural way to improve fault tolerance and provide high availability and reliability is therefore to construct several spanning trees whose communication paths share neither edges nor intermediate vertices. This can be achieved using completely independent spanning trees, which have therefore received considerable attention.

Let $G = (V,E)$ be a simple undirected graph with vertex and edge sets $V(G)$ and $E(G)$, respectively. The \textit{degree} of a vertex $x$ in $G$ is the number of vertices adjacent to it and is denoted $d_G(x)$. Let $x,y$ be two vertices of $G$. A $(x,y)$-\textit{path} in $G$ is a sequence of vertices from $x$ to $y$ such that each consecutive pair is an edge of $G$; it is a \textit{cycle} if $x = y$. $G$ is \textit{connected} if there is a path between any pair of vertices; it is $k$-\textit{connected} if it still connected even after deleting any set of $k-1$ vertices. The \textit{distance} between $x$ and $y$ in $G$, denoted by $d_G(x,y)$, is the length of the shortest $(x,y)$-path in $G$. The \textit{eccentricity} of $x$ in $G$, denoted by $e_G(x)$, is the maximum distance from $x$ to any other vertex of $G$. The \textit{diameter} of $G$, denoted by $\operatorname{diam}(G)$, is the maximum distance between any two vertices of $G$, equivalently, $\operatorname{diam}(G) = \max_{x \in V(G)} e_G(x)$. The \textit{radius} of $G$, denoted by $r(G)$, is the minimum eccentricity among all vertices of $G$. A vertex $x$ is called a \textit{center} of $G$ and denoted $c(G)$ if $e_G(x) = r(G)$. A graph can have multiple centers; however, a tree has exactly one center if its diameter is even, and exactly two adjacent centers otherwise. A graph is \textit{acyclic} if it contains no cycles. A \textit{tree} $T$ is a connected acyclic graph; it is a \textit{spanning tree} of $G$ if $V(T) = V(G)$ and $E(T) \subseteq E(G)$. A vertex $x$ is an \textit{internal} vertex of $T$ if $d_T(x) \geq 2$; otherwise, it is a \textit{leaf}.

A \textit{Hamiltonian} path of $G$ is a path that visits every vertex of $G$ exactly once. A Hamiltonian graph is a graph that contains a Hamiltonian cycle. Let $G$ be a graph and $P_1,P_2$ two $(x,y)$-paths. $P_1$ and $P_2$ are \textit{internally-disjoint} or vertex-disjoint paths if they do not share an internal vertex; they are \textit{openly-disjoint} if they are both internally-disjoint and edge-disjoint. Let $k \geq 2$ be an integer and $T_1,T_2,\dots,T_k$ be $k$ spanning trees of $G$. If for any pair of vertices $x,y$, the $(x,y)$-paths in $T_1,T_2,\dots,T_k$ are openly-disjoint, then $T_1,T_2,\dots,T_k$ are called \textit{completely independent spanning trees} (CIST for short) of $G$.

The study of CIST was originally motivated by applications in fault-tolerant communications and started with the work of Hasunuma \cite{2}, where CIST were introduced and characterized:
\begin{thm}[\cite{2}]
    Let $T_1, \dots, T_k$ be spanning trees in a graph $G$. Then, $T_1, \dots , T_k$ are completely independent if and only if $T_1, \dots , T_k$ are edge-disjoint and for any vertex $x \in V(G)$ and $1 \leq i \leq k$, there is at most one spanning tree $T_i$ such that $d_{T_i}(x) > 1.$
    \label{thm00}
\end{thm}

Hasunuma \cite{12} also proved that the problem of determining whether two completely independent spanning trees exist in a graph is NP-Complete and conjectured that any $2k$-connected graph contains $k$ CIST. However, this conjecture was later disproved \cite{4}. Following this result, several sufficient conditions derived from Hamiltonicity conditions have been studied to establish the existence of CIST \cite{5,40,22,73}. The CIST problem has also been investigated in specific graph classes such as maximal planar graphs, complete graphs, and split graphs \cite{12,32,lalou_completely_2026,3}. For more details about the completely independent spanning trees problem, readers may refer to \cite{82}.

In this paper, we first prove that the dual-cube $F_n$, a hypercube variant introduced by Li and Peng \cite{dual}, admits two completely independent spanning trees for every $n \geq 5$. Then we provide a recursive algorithm for constructing them. Moreover, the constructed trees have respective diameters of $6n-1$ and $6n+1$ for $n \geq 5$. The rest of this paper is organized as follows. Section~\ref{sec2} gives the preliminaries used across this paper. Section~\ref{sec3} proves the existence of two completely independent spanning trees for every $n \geq 5$ in the dual-cube $F_n$. Section~\ref{sec4} presents a recursive algorithm that builds on the results of the previous section and introduces an additional method to reduce the diameters of the respective trees. Section~\ref{sec5} concludes the paper with some research directions.
\section{Preliminaries on hypercube variants}
\label{sec2}
The $n$-dimensional hypercube $Q_n$ is the graph with $2^n$ vertices such that each vertex is labeled by an $n$-tuple from the set $\{ 0, 1 \}^n$; and two vertices are adjacent if and only if they differ in one coordinate. Thus, $Q_n$ is $n$-regular and $n$-connected.

In the literature, several results on CIST in hypercubes have been obtained. Pai et al. \cite{6} showed that $Q_n$ does not have $\lfloor \frac{n}{2}\rfloor $ CIST when $n$ is even and is not a power of two. Pai and Chang \cite{8} proved the existence of two CIST in $Q_n$ with a diameter equal to $2n-1$ for every $n\geq 4$. This result was later extended by Shaw \cite{78}, establishing that, for any positive integer $k$, $Q_n$ contains $2^k$ CIST when $n\geq 16\cdot 2^k-1$. Shaw also proved that these trees can be constructed with a diameter equal to $2+o(1))n$ whenever $n\geq 6\cdot 2^k+k+4$. More recently, Barabde et al. \cite{79} proved that $Q_n$ does not have $\lfloor \frac{n}{2} \rfloor$ CIST when $n$ is even with $2 < n \leq 10^7$, with some exceptions. They also showed that $Q_n$ has three CIST when $n \geq 7$, of diameters at most $2n+1, 2n+3$, and $2n+4$, respectively.

Many variants of the hypercube have been considered in the literature, such as the Möbius cubes \cite{8,70}, the augmented cubes \cite{11}, the balanced hypercubes \cite{14}, the divide-and-swap cubes \cite{51}, the shuffle cubes \cite{66,67}, and the bicubes \cite{50}.  Table \ref{variants} summarizes results on locally twisted cubes, crossed cubes, and parity cubes.
\begin{table}[h]
    \centering
    \begin{tabular}{|p{1.5cm}|p{12cm}|}
    \hline
    Variant & Result \\
    \hline 
    $LTQ_n$ & \cite{26} 2 CIST of diameter $2n - 2$ if $n = 4$, and $2n-3$ if $n \geq 5$. \newline     
        \cite{18}: 3 CIST of respective diameters $11,12$ and $12$ for $n = 6$ , and $2n -1$ when $n \geq 7$.\\
    \hline
    $CQ_n$ & \cite{7} 2 CIST of diameter $n+4$ when $n \geq 4$.\newline
             \cite{18}: 3 CIST of diameter $12$ when $n = 6$, and $2n +1$ when $n \geq 7$.\\
    \hline
    $PQ_n$ & \cite{8} 2 CIST of diameter $2n-1$ when $n \geq 4$. \\ 
    \hline
    \end{tabular}
    \caption{Known results on CIST in hypercube variants}
    \label{variants}
\end{table}

As the hypercube has a relatively large number of edges for a given number of vertices, Li and Peng \cite{dual} introduced the dual-cube, which preserves a hypercube-based structure while significantly reducing the number of edges. More specifically, for a fixed dimension $n$ (fixed number of edges per vertex), the dual-cube has $2^{2n-1}$ vertices, which is $2^{n-1}$ times more than the $2^n$ vertices of the hypercube. Formally, the dual-cube $F_n$ is an $n$-regular $n$-connected graph on the vertex set $\{0,1\}^{2n-1}$ such that two vertices $u = (u_{2n-1},\dots, u_1)$ and $v = (v_{2n-1},\dots, v_1)$ are adjacent in $F_n$ if and only if:
\begin{enumerate}
    \item $u$ and $v$ differ only in one bit position $i$.
    \item if $1 \leq i \leq n-1$ then $u_{2n-1} = v_{2n-1} = 0$.
    \item if $n \leq i \leq 2n-2$ then $u_{2n-1} = v_{2n-1} = 1$.
\end{enumerate}
We note that the dual-cube $F_n$ is composed of $(n-1)$-dimensional hypercubes $Q_{n-1}$ as basic components. These components are called clusters and are evenly distributed across the two classes $0$ and $1$. A cluster of class $0$ is a set of vertices $u$ of form $(0,u_{2n-2},\dots, u_n, * ,\dots, *)$ where $*$ is an arbitrary value. In the same way, a cluster of class $1$ is a set of vertices $u$ of form $(1 ,*, \dots ,* ,u_{n-1},\dots, u_1)$. Clusters of the same class are not connected. An edge that connects two vertices of distinct classes is called a {\em cross-edge}. This labeling of the dual-cube vertices implies a binary representation of a vertex $u = (u_{2n-1},\dots, u_1)$ into three parts:

\begin{table}[h]
\centering
\begin{tabular}{|c|c|c|c|}
\hline
Cluster class & $u_{2n-1}$ & $u_{2n-2} \dots u_{n}$ & $u_{n-1} \dots u_1$  \\
\hline
Class $0$ & $0$ & Node ID & Cluster ID \\
\hline
Class $1$ & $1$ & Cluster ID & Node ID \\
\hline
\end{tabular}
\label{labeling}
\end{table}
Moreover, in what follows, clusters (respectively, vertices) are labeled using the decimal representation of their cluster (respectively, vertex) IDs. By the definition of the dual-cube structure, the vertex $a$ of the cluster $b$ in class $i$ is connected to the vertex $b$ of the cluster $a$ in class $1-i$. Note also that each cluster $C_1$ of class $i$ has exactly one cross-edge with every cluster $C_2$ of class $1-i$.
The dual-cube $F_n$ can also be constructed recursively using four copies of $F_{n-1}$ as follows \cite{dual}: Let $F_{n-1}^i$ denote the $i^{th}$ copy of $F_{n-1}$, where $i=b_2b_1$ with $b_1,b_2\in\{0,1\}$. For $a\in\{0,1\}$, let $S_a^{b_2b_1}$ be the set of clusters of class $a$ in the copy $F_{n-1}^{b_2b_1}$. This means that, for a vertex $(a x_{n-2} \dots x_{1} y_{n-2} \dots y_1)$ in $F_{n-1}$, its vertex address in $F_n$ becomes $(a b_2 x_{n-2} \dots x_{1} b_1 y_{n-2} \dots y_1)$ with $0 \leq b_2 b_1 \leq 3$. Then, clusters of class $0$ in $F_n$ are formed by connecting, pair by pair, the vertices of $S_0^{00}$ with the vertices of $S_0^{01}$, and the vertices of $S_0^{10}$ with the vertices of $S_0^{11}$. In the same way, clusters of class $1$ in $F_n$ are formed by connecting, pair by pair, the vertices of $S_1^{00}$ with the vertices of $S_1^{10}$, and the vertices of $S_1^{01}$ with the vertices of $S_1^{11}$. See Figure~\ref{recursive} for an illustration of the recursive construction process.
\begin{figure}[h]
    \centering
\begin{tikzpicture}
    \node[shape=circle] (1) at (0,0) {$S_0^{00}$};
    \node[shape=circle] (2) at (2,0) {$S_1^{00}$};
    \node[shape=circle] (3) at (6,0) {$S_0^{01}$};
    \node[shape=circle] (4) at (8,0) {$S_1^{01}$};
    \node[shape=circle] (5) at (0,-3) {$S_0^{11}$};
    \node[shape=circle] (6) at (2,-3) {$S_1^{11}$};
    \node[shape=circle] (7) at (6,-3) {$S_0^{10}$};
    \node[shape=circle] (8) at (8,-3) {$S_1^{10}$};
    \node at (-1.4,0) {$F^{00}_{n-1}$};
    \node at (9.4,0) {$F^{01}_{n-1}$};
    \node at (9.4,-3) {$F^{10}_{n-1}$};
    \node at (-1.4,-3) {$F^{11}_{n-1}$};
    
    \path [-] (2) edge node {} (8);
    \path [-] (4) edge node {} (6);
    \draw[-] (1) to[out=30, in= 140] (3);
    \draw[-] (5) to[out=30, in= 140] (7);

    \begin{scope}[fill opacity=0,dashed]
    \filldraw[] ($(1)+(-0.5,0.5)$)
        to[out=0,in=-180] ($(2) + (0.5,0.5)$)    
        to[out=0,in=0] ($(2) + (0.5,-0.5)$)
        to[out=180,in=0] ($(1) + (-0.5,-0.5)$)
        to[out=180,in=180] ($(1) + (-0.5,0.5)$) ;
    \end{scope}    

    \begin{scope}[fill opacity=0,dashed]
    \filldraw[] ($(3)+(-0.5,0.5)$)
        to[out=0,in=-180] ($(4) + (0.5,0.5)$)    
        to[out=0,in=0] ($(4) + (0.5,-0.5)$)
        to[out=180,in=0] ($(3) + (-0.5,-0.5)$)
        to[out=180,in=180] ($(3) + (-0.5,0.5)$) ;
    \end{scope}    

    \begin{scope}[fill opacity=0,dashed]
    \filldraw[] ($(5)+(-0.5,0.5)$)
        to[out=0,in=-180] ($(6) + (0.5,0.5)$)    
        to[out=0,in=0] ($(6) + (0.5,-0.5)$)
        to[out=180,in=0] ($(5) + (-0.5,-0.5)$)
        to[out=180,in=180] ($(5) + (-0.5,0.5)$) ;
    \end{scope}    

    \begin{scope}[fill opacity=0,dashed]
    \filldraw[] ($(7)+(-0.5,0.5)$)
        to[out=0,in=-180] ($(8) + (0.5,0.5)$)    
        to[out=0,in=0] ($(8) + (0.5,-0.5)$)
        to[out=180,in=0] ($(7) + (-0.5,-0.5)$)
        to[out=180,in=180] ($(7) + (-0.5,0.5)$) ;
    \end{scope}    
\end{tikzpicture}
        \caption{Recursive construction of $F_n$}
        \label{recursive}
\end{figure}
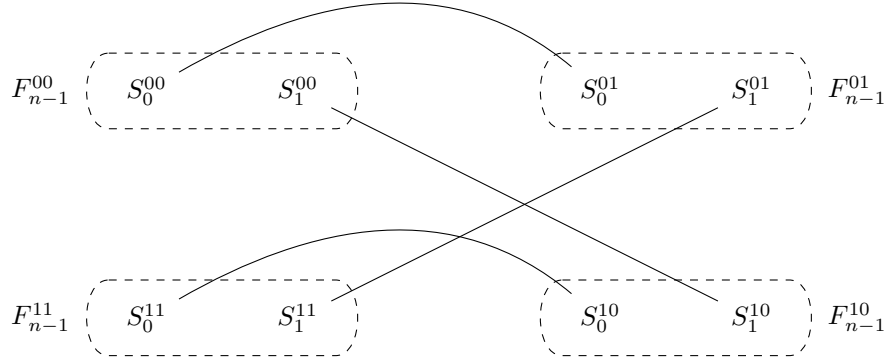

Several structural properties of the dual-cube have been studied. Table~\ref{properties}, which is inspired by \cite{dual}, shows a comparison between the hypercube and the dual-cube with the same number of vertices, $2^n$, for some properties. Moreover, the dual-cube has strong Hamiltonian properties. In \cite{193}, it was shown that the $n$-dimensional dual-cube $F_n$ remains hamiltonian even in the presence of up to $n-1$ faulty edges (i.e., edges affected by failures and therefore unavailable), and that this bound is optimal. In \cite{188p}, the authors proved that it contains $\lfloor \frac{n-1}{2} \rfloor$ edge-disjoint hamiltonian cycles. Furthermore, in \cite{188}, it was proved that the dual-cube contains $n$ internally-disjoint hamiltonian cycles for $n\geq 2$, and this result is optimal. Since several sufficient conditions for the existence of CIST are based on Hamiltonicity, the Hamiltonian properties of the dual-cube naturally motivate the study of the existence of multiple CIST.

\begin{table}[h]
    \centering
    \begin{tabular}{|c|c|c|c|c|c|}
    \hline
    Graph & Degree & \# links & Diameter & Avg. distance & Bisection width \\
    \hline
    $Q_n$ & $n$ & $2^nn$ & $n$ & $\frac{n}{2}$ & $\frac{2^n}{2}$ \\
    \hline
    $F_{\frac{n+1}{2}}$ & $\frac{n+1}{2}$ & $\frac{2^n(n+1)}{2}$ & $n+1$ & $\frac{n}{2} + 1 - \frac{1}{2^{\frac{n-1}{2}}}$ & $\frac{2^n}{4}$ \\
    \hline
    \end{tabular}
    \caption{Comparison of some properties of the hypercube and the dual-cube}
    \label{properties}
\end{table}

\section{Existence of two CIST in $F_n$}
\label{sec3}
Although the Hamiltonian properties cited in Section~\ref{sec2} suggest the existence of CIST in the dual-cube, they do not directly provide the constructions required. Therefore, in this section, we prove that the dual-cube $F_n$ contains two completely independent spanning trees for every $n \geq 5$. To establish this result, we use a cluster-based approach. More precisely, we determine suitable inter-cluster connections in the contracted graph (see definition below) and then show that appropriate local CIST can complete these connections within each cluster. The contracted graph $\mathrm{KF}_n$ is obtained from $F_n$ by:
\begin{itemize}
    \item replacing each cluster of $F_n$ by a vertex in $\mathrm{KF}_n$.
    \item connecting any pair of vertices of $\mathrm{KF}_n$ if and only if there exists at least one cross-edge in $F_n$ between their two corresponding clusters.
\end{itemize}

For $n \leq 3$, the dual-cube $F_n$ does not contain two CIST. Indeed, since $F_n$ is $n$-regular, it has $n2^{2n-2}$ edges; however, two edge-disjoint spanning trees require at least $2(2^{2n-1}-1)$ edges. This necessary condition is not satisfied for $n \leq3 $, making $F_4$ the first unresolved case. We therefore attempted to find two CIST in $F_4$ using the ILP model in Appendix \ref{appB}, inspired by \cite{34}, and implemented in Python with the Gurobi optimizer. The computation was performed on a machine equipped with an Intel(R) Xeon(R) CPU E5-2609 0 @ 2.40 GHz processor, 32 GB of memory, and running Windows 10 Professional 64-bit. However, even after more than 700 hours of computation, no conclusive answer was obtained. 

Subsequently, we were aware that the case $F_4$ had been resolved by Jitendra Prajapati, who presented a construction of two CIST in \(F_4\)~\cite{Prajapati2026}. A direct computation shows that the two trees have diameters $30$ and $31$, respectively. If these trees are used as the base case of the recursive construction described in the proof of Theorem~\ref{thm6}, the resulting CIST in $F_n$ have respective diameters of $6n+5$ and $6n+7$ for every $n\geq5$. These diameters are larger than the diameters $6n-1$ and $6n+1$ obtained by our construction based on $F_5$. Therefore, in the remainder of this paper, we continue with our $F_5$-based construction. For this reason, we exploit the cluster structure of $F_n$ and decompose the global CIST into local ones within the clusters of the dual-cube. 

The local CIST of each cluster isomorphic to $Q_{n-1}$ cannot be chosen arbitrarily. Indeed, the clusters are connected through cross-edges, and each cross-edge used for each global tree must be incident on vertices that are internal in the two corresponding local trees, \textit{i.e.}, if the internal vertices inside the clusters are not chosen carefully, then the cross-edges may fail to connect all local CIST into two global spanning trees of $F_n$. Hence, the local CIST within each cluster must be constructed with internal vertices selected from the cross-edges connecting the clusters. Therefore, instead of constructing the local CIST independently in each cluster, we first use the contracted graph $\mathrm{KF}_n$ to determine how the clusters should be connected. Then, for each cluster, the local CIST are constructed so that their internal vertices are compatible with the selected cross-edges. This guarantees that the local constructions can be assembled into two connected global spanning trees of the whole dual-cube. 

Since each cluster of class $0$ is connected by one cross-edge to each cluster of class $1$, the graph $\mathrm{KF}_n$ is isomorphic to $K_{2^{n-1},2^{n-1}}$. The first goal is to construct two edge-disjoint spanning trees, $\mathrm{KT}_1$ and $\mathrm{KT}_2$, in $\mathrm{KF}_n$, which are later transformed into two CIST, $T_1$ and $T_2$, in $F_n$. In more detail, since the clusters of $F_n$ are isomorphic to the $(n-1)$-dimensional hypercube, two CIST are constructed for each cluster $\overline{C}^{b}$ such that $b =0,1$ represents the class, and $C=0,\dots,2^{n-1}$ represents the number of the cluster in the class $C$, while at the same time guaranteeing that the trees in all clusters can be connected by cross edges in the respective trees $T_1$ and $T_2$. To achieve that, let us introduce the following lemmas.
\begin{lemme}
    For $n \geq 4$ and for every two disjoint unordered pairs of vertices $\{u,v\}$ and $\{x,y\}$ of $Q_n$, two completely independent spanning trees $T_1$ and $T_2$ can be constructed with $\{u,v\}$ ($\{x,y\}$, respectively) as internal vertices of $T_1$ ($T_2$, respectively). 
\label{lemme1}
\end{lemme}
\begin{proof}
We prove the lemma by induction on $n$, with an exhaustive verification for the base case $n = 4$, and an induction step showing how to extend the construction from $Q_n$ to $Q_{n+1}$:

\textit{Verification of the base case $n = 4$}. For every choice of two disjoint pairs of vertices $\{u,v\}$ and $\{x,y\}$ in $Q_n$, the ILP program of Appendix \ref{appB} is run after adding the constraints:
    \[
    y_{1u}=y_{1v}=1,\qquad y_{2x}=y_{2y}=1.
    \]
These constraints force $u$ and $v$ to be internal vertices of the tree $T_1$, and $x$ and $y$ to be internal vertices of the tree $T_2$. Since the hypercube is vertex-transitive, we can fix one vertex of the first pair. Thus, without loss of generality, we set $u = 0000$ and enumerate all possible choices of the remaining vertices $v,x,y$, while maintaining the previous condition that the two pairs $\{u,v\}$ and $\{x,y\}$ are disjoint. Therefore, the number of instances to be checked is:
    \[
    \binom{15}{3}=455.
    \]
In the following, vertices are labeled using the decimal representation of their binary addresses. Thus, the vertex $u=0000$ is labeled $0$.
For each instance, the ILP found a feasible solution. Therefore, the case $n = 4$ holds. As an illustration, consider the following two disjoint pairs: 
\[
\{u,v\}=\{0,7\}
\qquad \text{and} \qquad
\{x,y\}=\{3,10\},
\]
The ILP formulation returns two CIST, in which $u$ and $v$ are internal vertices of $T_1$, and $x$ and $y$ are internal vertices of $T_2$, as shown in Figure~\ref{Q4}.
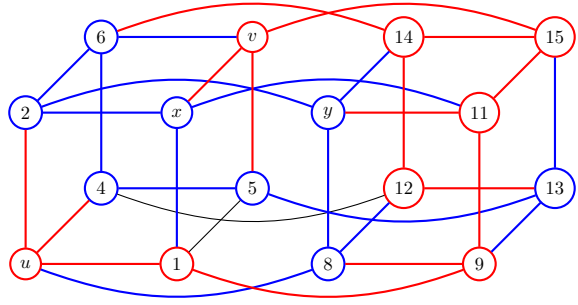
\begin{figure}[h]
    \centering
    \begin{tikzpicture}   
\node[shape=circle,draw=red,thick,scale=0.7] (1) at (-10,1) {$u$};
\node[shape=circle,draw=red,thick,scale=0.7] (2) at (-8,1) {1};
\node[shape=circle,draw=blue,thick,scale=0.7] (3) at (-10,3) {2};
\node[shape=circle,draw=blue,thick,scale=0.7] (4) at (-8,3) {$x$};
\node[shape=circle,draw=blue,thick,scale=0.7] (5) at (-9,2) {4};
\node[shape=circle,draw=blue,thick,scale=0.7] (6) at (-7,2) {5};
\node[shape=circle,draw=blue,thick,scale=0.7] (7) at (-9,4) {6};
\node[shape=circle,draw=red,thick,scale=0.7] (8) at (-7,4) {$v$};
\node[shape=circle,draw=blue,thick,scale=0.7] (9) at (-6,1) {8};
\node[shape=circle,draw=red,thick,scale=0.7] (10) at (-4,1) {9};  
\node[shape=circle,draw=blue,thick,scale=0.7] (11) at (-6,3) {$y$};
\node[shape=circle,draw=red,thick,scale=0.7] (12) at (-4,3) {11};
\node[shape=circle,draw=red,thick,scale=0.7] (13) at (-5,2) {12};
\node[shape=circle,draw=blue,thick,scale=0.7] (14) at (-3,2) {13};
\node[shape=circle,draw=red,thick,scale=0.7] (15) at (-5,4) {14};
\node[shape=circle,draw=red,thick,scale=0.7] (16) at (-3,4) {15};

\draw[-,thick,red] (7) to[out=20, in=160] (15);
\draw[-,thick,red] (8) to[out=20, in=160] (16);
\draw[-,thick,blue] (3) to[out=20, in=160] (11);
\draw[-,thick,blue] (4) to[out=20, in=160] (12);
\draw[-] (5) to[out=-20, in=-160] (13);
\draw[-,thick,blue] (6) to[out=-20, in=-160] (14);
\draw[-,thick,blue] (1) to[out=-20, in=-160] (9);
\draw[-,thick,red] (2) to[out=-20, in=-160] (10);

\path[-,thick,blue] (7) edge node {} (8);
\path[-,thick,red] (15) edge node {} (16);
\path[-,thick,blue] (3) edge node {} (4);
\path[-,thick,red] (11) edge node {} (12);
\path[-,thick,blue] (5) edge node {} (6);
\path[-,thick,red] (1) edge node {} (2);
\path[-,thick,red] (13) edge node {} (14);
\path[-,thick,red] (10) edge node {} (9);
\path[-,thick,blue] (7) edge node {} (5);
\path[-,thick,red] (6) edge node {} (8);
\path[-,thick,red] (3) edge node {} (1);
\path[-,thick,blue] (4) edge node {} (2);
\path[-,thick,blue] (11) edge node {} (9);
\path[-,thick,red] (10) edge node {} (12);
\path[-,thick,red] (15) edge node {} (13);
\path[-,thick,blue] (14) edge node {} (16);
\path[-,thick,blue] (7) edge node {} (3);
\path[-,thick,red] (4) edge node {} (8);
\path[-,thick,red] (5) edge node {} (1);
\path[-] (6) edge node {} (2);
\path[-,thick,blue] (11) edge node {} (15);
\path[-,thick,red] (16) edge node {} (12);
\path[-,thick,blue] (9) edge node {} (13);
\path[-,thick,blue] (14) edge node {} (10);
    \end{tikzpicture}
        \caption{Two CIST of $Q_4$ with the internal vertices $\{u =0,v = 7\}$ for $T_1$ and $\{x = 3,y = 10\}$ for $T_2$.}
        \label{Q4}
\end{figure}

\noindent\textit{Extention of the construction from $Q_n$ to $Q_{n+1}$}. Assume that the statement holds for $Q_n$ for some $n\geq 4$. We prove that it also holds for $Q_{n+1}$: Recall that, for any $1 \leq i \leq n$, $Q_{n+1}$ can be decomposed into two copies of $Q_n$, denoted by $Q_n^{i,0}$ and $Q_n^{i,1}$, such that $Q_n^{i,j}$ contains the vertices whose $i^{th}$ coordinate is $j$ for $j =0,1$. For every vertex $z\in V(Q_{n+1})$, we denote by $z^{i,0}$ and $z^{i,1}$ its corresponding copies in $Q_n^{i,0}$ and $Q_n^{i,1}$ respectively. The edge $(z^{i,0},z^{i,1})$ is called an \textit{$(n+1)$-dimensional} edge. Moreover, we write \(z=(z_{n+1},\ldots,z_1)\) for every
$ z\in V(Q_{n+1})$.

Let $\{u,v\}$ and $\{x,y\}$ be two disjoint pairs of vertices of $Q_{n+1}$. The goal is to construct two completely independent spanning trees $T_1$ and $T_2$ such that $u$ and $v$ are internal vertices of $T_1$, and $x$ and $y$ are internal vertices of $T_2$. The proof is divided into four main cases, depending on the positions of the vertices in the two copies $Q_n^{i,0}$ and $Q_n^{i,1}$ for a value of $i$. By symmetry between the two copies $Q_n^{i,0}$ and $Q_n^{i,1}$, and by exchanging $u$ with $v$, or $x$ with $y$, the following cases cover all possible distributions of the four vertices between the two copies. Indeed, the distribution is either of type $4+0$, $3+1$, or $2+2$. In the $2+2$ case, either the two vertices for the same tree are in the same copy, or each copy contains one vertex for each tree.

\vspace{0.2cm}
\noindent\textit{Case 1.} $ \exists \ i$ such that $u,v,x,y \in V(Q_n^{i,0})$. By the induction hypothesis, two completely independent spanning trees $T_1^0$ and $T_2^0$ can be constructed in $Q_n^{i,0}$ such that $u^{i,0}$ and $v^{i,0}$ are internal vertices of $T_1^0$, while $x^{i,0}$ and $y^{i,0}$ are internal vertices of $T_2^0$. Now consider the second copy $Q_n^{i,1}$. We construct two spanning trees, $T_1^1$ and $T_2^1$ in $Q_n^1$, by taking the isomorphic copies of $T_1^0$ and $T_2^0$, respectively. Hence, $u^{i,1}$ and $v^{i,1}$ are internal vertices of $T_1^1$, while $x^{i,1}$ and $y^{i,1}$ are internal vertices of $T_2^1$. We define:
\[
T_1 = T_1^0 \cup T_1^1 \cup (u^{i,0},u^{i,1})),
\]
and
\[
T_2 = T_2^0 \cup T_2^1 \cup (x^{i,0},x^{i,1}).
\]
Since $T_1^0$ and $T_1^1$ are spanning trees of the two copies $Q_n^{i,0}$ and $Q_n^{i,1}$, adding these edges connects them and creates a spanning tree of $Q_{n+1}$, the same applies for $T_2^0$ and $T_2^1$. Thus, $T_1$ and $T_2$ are two CIST of $Q_{n+1}$.

\vspace{0.2cm}
\noindent\textit{Case 2.} $ \exists \ i$ such that $u \in V(Q_n^{i,0})$ and $v,x,y \in V(Q_n^{i,1})$. We first choose a random vertex $a^{i,1} \in V(Q_n^{i,1})\setminus\{v^{i,1},x^{i,1},y^{i,1} \}$. By the induction hypothesis, we construct two CIST in $Q_n^{i,1}$, such that $v^{i,1}$ and $a^{i,1}$ are internal vertices of $T_1^{i,1}$, while $x^{i,1}$ and $y^{i,1}$ are internal vertices of $T_2^{i,1}$. Then, we distinguish three subcases according to the $(n+1)$-dimensional neighbor of $u^{i,0}$:

\vspace{0.2cm}
\noindent\textit{Subcase 2.1.} $u^{i,0}$ and $v^{i,1}$ are $(n+1)$-dimensional neighbors. Let $a^{i,0}$, $x^{i,0}$, and $y^{i,0}$ be the $(n+1)$-dimensional neighbors of $a^{i,1}$, $x^{i,1}$, and $y^{i,1}$ in $Q_n^{i,0}$, respectively. In $Q_n^{i,0}$, we construct two CIST such that $u^{i,0}$ and $a^{i,0}$ are internal vertices of $T_1^0$, while $x^{i,0}$ and $y^{i,0}$ are internal vertices of $T_2^0$. We then connect the two copies using the edges $(u^{i,0},v^{i,1})$ for $T_1$ and $(x^{i,0},x^{i,1})$ for $T_2$. The resulting trees are two CIST of $Q_{n+1}$.

\vspace{0.2cm}
\noindent\textit{Subcase 2.2.} $u^{i,0}$ is an $(n+1)$-dimensional neighbor of one of $x^{i,1}$ and $y^{i,1}$. Without loss of generality, assume that $u^{i,0}$ and $x^{i,1}$ are $(n+1)$-dimensional neighbors. Let $v^{i,0}$ and $y^{i,0}$ be the $(n+1)$-dimensional neighbors of $v^{i,1}$ and $y^{i,1}$ in $Q_n^{i,0}$, respectively. We construct two CIST in $Q_n^{i,0}$ such that $u^{i,0}$ and $v^{i,0}$ are internal vertices of $T_1^0$, while $y^{i,0}$ and a random vertex $b^{i,0} \notin \{ u^{i,0},v^{i,0},y^{i,0} \}$ are internal vertices of $T_2^0$. We then connect the two copies using the edge $(v^{i,0},v^{i,1})$ for $T_1$ and the edge $(y^{i,0},y^{i,1})$ for $T_2$. The resulting trees are two CIST of $Q_{n+1}$.
    
\vspace{0.2cm}
\noindent\textit{Subcase 2.3.} $u^{i,0}$ is not a $(n+1)$-dimensional neighbor of any vertex among $v^{i,1},x^{i,1}$ and $y^{i,1}$. Let $v^{i,0}$, $x^{i,0}$, and $y^{i,0}$ be the $(n+1)$-dimensional neighbors of $v^{i,1}$, $x^{i,1}$, and $y^{i,1}$ in $Q_n^{i,0}$, respectively. In this case, we construct two CIST in $Q_n^{i,0}$ such that $u^{i,0}$ and $v^{i,0}$ are internal vertices of $T_1^0$, while $x^{i,0}$ and $y^{i,0}$ are internal vertices of $T_2^0$. We then connect the two copies using the edge $(v^{i,0},v^{i,1})$ for $T_1$ and the edge $(y_{i,0},y^{i,1})$ for $T_2$. The resulting trees are two CIST of $Q_{n+1}$.

\vspace{0.2cm}
\noindent\textit{Case 3.} $\exists i$ such that $u,x \in V(Q_n^{i,0})$ and $v,y \in V(Q_n^{i,1})$. In this case, notice that if $u^{i,0}$ and $v^{i,1}$ are $(n+1)$-dimensional neighbors, then the edge $(u^{i,0},v^{i,1})$ can be used to connect the two copies for $T_1$. Similarly, if $x^{i,0}$ and $y^{i,1}$ are $(n+1)$-dimensional neighbors, then the edge $(x^{i,0},y^{i,1})$ can be used to connect the two copies for $T_2$. Therefore, we distinguish three subcases according to the $(n+1)$-dimensional neighbors among the cross-pairs $(u^{i,0},y^{i,1})$ and $(x^{i,0},v^{i,1})$.

\vspace{0.2cm}
\noindent\textit{Subcase 3.1.} Neither $(u^{i,0},y^{i,1})$ nor $(x^{i,0},v^{i,1})$ is an $(n+1)$-dimensional edge. We choose a random vertex $a^{i,0}\in V(Q_n^{i,0})\setminus\{u^{i,0},x^{i,0},v^{i,0},y^{i,0} \}$, and denote by $a^{i,1}$ its $(n+1)$-dimensional neighbor in $Q_n^{i,1}$. Similarly, we choose a random vertex $b^{i,1}\in V(Q_n^{i,1})\setminus\{v^{i,1},y^{i,1},u^{i,1},x^{i,1},a^{i,1}\}$, and denote by $b^{i,0}$ its $(n+1)$-dimensional neighbor in $Q_n^{i,0}$. By the induction hypothesis, we construct two CIST in $Q_1^{i,0}$ such that $u^{i,0}$ and $a^{i,0}$ are internal vertices of $T_1^0$, and $x^{i,0}$ and $b^{i,0}$ are internal vertices of $T_2^0$. Similarly, we construct two CIST in $Q_n^{i,1}$ such that $v^{i,1}$ and $a^{i,1}$ are internal vertices of $T_1^1$, while $y^{i,1}$ and $b^{i,1}$ are internal vertices of $T_2^1$. We then connect the two copies by using the edge $(a^{i,0},a^{i,1})$ for $T_1$ and the edge $(b^{i,0},b^{i,1})$ for $T_2$. Therefore, the resulting trees are two CIST of $Q_{n+1}$. Notice that possible overlaps, such as $u^{i,0}=v^{i,0}$ or $x^{i,0}=y^{i,0}$, do not affect the construction. This is because we do not use the edge $(u^{i,0},v^{i,1})$ nor $(x^{i,0},y^{i,1})$ to connect the two copies.

\vspace{0.2cm}
\noindent\textit{Subcase 3.2.} Exactly one of $(u,y)$ and $(x,v)$ is a $(n+1)$-dimensional edge. Without loss of generality, assume that $u^0$ and $y^1$ are $(n+1)$-dimensional neighbors. Then $x^0$ and $v^1$ are not $(n+1)$-dimensional neighbors. Let $v^0$ be the $(n+1)$-dimensional neighbor of $v^1$ in $Q_n^0$, and let $x^1$ be the $(n+1)$-dimensional neighbor of $x^0$ in $Q_n^1$. We have $v^0\neq x^0$ and $x^1\neq v^1$. We construct two CIST in $Q_n^0$ such that $u^0$ and $v^0$ are internal vertices of $T_1^0$, while $x^0$ and a random vertex $a^0 \notin \{ u^0,x^0,v^0 \}$ are internal vertices of $T_2^0$. Similarly, we construct two CIST in $Q_n^1$ such that $v^1$ and a random vertex $b^1 \notin \{ v^1,y^1,x^1 \}$ are internal vertices of $T_1^1$, while $y^1$ and $x^1$ are internal vertices of $T_2^1$. We then connect the two copies using the edge $(v^0,v^1)$ for $T_1$ and the edge $(x^0,x^1)$ for $T_2$. The resulting trees are two CIST of $Q_{n+1}$.

\vspace{0.2cm}
\noindent\textit{Subcase 3.3.} Both $(u,y)$ and $(x,v)$ are $(n+1)$-dimensional edges. Let $a^{i,0},b^{i,0} \notin \{ u^{i,0},x^{i,0} \}$ in $Q_n^{i,0}$, and let $a^{i,1},b^{i,1} \notin \{ v^{i,1},y^{i,1} \}$ in $Q_n^{i,1}$ be the $(n+1)$-dimensional neighbors of $a^{i,0},b^{i,0}$, respectively. In $Q_n^{i,0}$, we construct two CIST such that $u^{i,0}$ and $a^{i,0}$ are internal vertices of $T_1^0$, while $x^{i,0}$ and $b^{i,0}$ are internal vertices of $T_2^0$. Similarly, in $Q_n^{i,1}$, we construct two CIST such that $y^{i,1}$ and $a^{i,1}$ are internal vertices of $T_1^1$, while $v^{i,1}$ and $b^{i,1}$ are internal vertices of $T_2^1$. We then connect the two copies using the edge $(a^{i,0},a^{i,1})$ for $T_1$ and the edge $(b^{i,0},b^{i,1})$ for $T_2$. The resulting trees are two CIST of $Q_{n+1}$.

\vspace{0.2cm}
\noindent\textit{Case 4.} Assume that none of the above cases applies. This means that for each position $i$, $1\le i\le n$, none of the following cases occurs:
\[
u_i = v_i = x_i =y_i,\qquad
u_i \neq v_i= x_i =y_i,\qquad
u_i=x_i \neq v_i=y_i.
\]
Therefore, for each position $i$, we have $u_i=v_i \neq x_i=y_i$. This implies $u=v$, contradicting the hypothesis that $u,v,x$ and $y$ are all distinct.
\end{proof}

In what follows, the vertices $\{u,v\}$ ($\{x,y\}$, respectively) that satisfy Lemma~\ref{lemme1} are called {\em generators} of $T_1$ ($T_2$, respectively).

\begin{lemme}
    Let $n \geq 5$. $\mathrm{KF}_n$ has at least two edge-disjoint Hamiltonian paths.
    \label{lemme2}
\end{lemme}
\begin{proof}
Recall that $\mathrm{KF}_n \equiv K_{2^{n-1},2^{n-1} }$. It is well known that $K_{n,n}$ contains at least two edge-disjoint Hamiltonian paths, since it admits $\lfloor n/2\rfloor$ edge-disjoint Hamiltonian cycles \cite{laskar_decomposition_1976}. For the sake of completeness, we give the following construction: Let $m=2^{n-1}$ and $p$ be an odd number. Since $m$ is a power of two, we have $\gcd(p,m)=1$. Therefore, the sequence:
\[
x,\ x+p,\ x+2p,\ \dots,\ x+(m-1)p
\]
contains every value modulo $m$ exactly once. We define the path:
\[
P^1_p= \overline{1}^0,  \overline{p+1}^1,  \overline{2p+1}^0,  \overline{3p+1}^1,  \overline{4p+1}^0, \dots,  \overline{(m-1)p+1}^1,
\]
with all indices taken modulo $m$. Since $m$ is even and $p$ is odd, this first part visits exactly $\frac{m}{2}$ clusters of class $0$ and $\frac{m}{2}$ clusters of class $1$. More precisely, the class $0$ indices are the odd residues modulo $m$, while the class $1$ indices are the even residues modulo $m$. Therefore, to visit the remaining clusters, a shift of one position is introduced in the second subpath (\textit{i.e.} we start from $C_2^0$):
\[
P^2_p=  \overline{2}^0,  \overline{p+2}^1,  \overline{2p+2}^0,  \overline{3p+2}^1,  \overline{4p+2}^0, \dots,  \overline{(m-1)p+2}^1.
\]
The second part visits the remaining $\frac{m}{2}$ clusters of class $0$ and the remaining $\frac{m}{2}$ clusters of class $1$. Consequently, the path $P_p = P^1_p \cup P^2_p \cup \{ \overline{(m-1)p+1}^1, \overline{2}^0\}$ visits every cluster of $\mathrm{KF}_n$ exactly once. Hence, $P_p$ is a Hamiltonian path of $\mathrm{KF}_n$. By choosing two distinct prime numbers $p$ and $q$ that range between 1 and $m-1$, and applying the construction above for each prime number, two edge-disjoint Hamiltonian paths $P_p$ and $P_q$ are obtained. The construction of the path $P_p$ is illustrated in Figure~\ref{pathconstruction}, where the first part follows the modular progression with step $p$, and the second part uses the same progression shifted by one position.

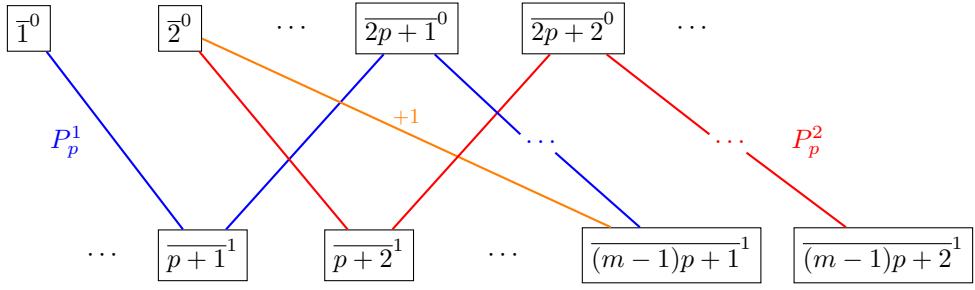
\begin{figure}[h]
\centering
\begin{tikzpicture}

\node[draw=black,shape = rectangle] (c01) at (0,0) {$\overline{1}^0$};
\node[draw=black,shape = rectangle] (c02) at (2,0) {$ \overline{2}^0$};
\node at (3.5,0) {$\cdots$};
\node[draw=black,shape = rectangle] (c03p1) at (5,0) {$ \overline{2p+1}^0$};

\node[draw=black,shape = rectangle] (c03p2) at (7.2,0) {$ \overline{2p+2}^0$};
\node at (8.8,0) {$\cdots$};

\node at (1,-3) {$\cdots$};
\node[draw=black,shape = rectangle] (c12p1) at (2.3,-3) {$ \overline{p+1}^1$};
\node[draw=black,shape = rectangle] (c12p2) at (4.5,-3) {$ \overline{p+2}^1$};
\node at (6.3,-3) {$\cdots$};
\node[draw=black,shape = rectangle] (c1m2p1) at (8.5,-3) {$ \overline{(m-1)p+1}^1$};

\node[draw=black,shape = rectangle] (c1m1p1) at (11.3,-3) {$ \overline{(m-1)p+2}^1$};

    \path[-,thick,blue] (c01) edge (c12p1);
    \path[-,thick,blue] (c12p1) edge (c03p1);

    \draw[-,thick,orange] (c02) -- node[midway,above,scale = 0.8] {$+1$} (c1m2p1);
    \path[-,thick,red] (c02) edge (c12p2);
    \path[-,thick,red] (c03p2) edge (c12p2);

\node[blue] (pp1) at (0.5,-1.5) {$P_p^1$};
\node[red] (pp2) at (10.3,-1.5) {$P_p^2$};
\node[blue] (dot1) at (6.8,-1.5) {$\dots$};
\node[red] (dot2) at (9.3,-1.5) {$\dots$};

    \path[-,thick,blue] (c03p1) edge (6.5,-1.4);
    \path[-,thick,blue] (c1m2p1) edge (dot1);

    \path[-,thick,red] (c03p2) edge (9,-1.4);
    \path[-,thick,red] (c1m1p1) edge (dot2);
    
\end{tikzpicture}
\caption{Illustration of the construction of $P_p$ in $\mathrm{KF}_n$}
\label{pathconstruction}
\end{figure}

\end{proof}

Using Lemma~\ref{lemme1} and Lemma~\ref{lemme2}, we obtain the following result.
\begin{thm}
    Let $n \geq 5$. The dual-cube $F_n$ has two CIST.
    \label{thm0}
\end{thm}
\begin{proof}
Assume that $P_1$ and $P_2$ are two edge-disjoint Hamiltonian paths of $\mathrm{KF}_n$ constructed using Lemma \ref{lemme2}. The edges of $P_1$ and $P_2$ incident on each cluster $C$ determine which vertices of $C$ must be internal in the local trees. 
For instance, if the cross edges of some $P_i$ connect cluster $\overline{\ell}^0$ with cluster $\overline{k}^1$ and $\overline{m}^1$, then the two vertices $k$ and $m$ have to be internal in the tree $T_i$ of cluster $\overline{\ell}^0$. Figure~\ref{Pi} illustrates how the subpath $\overline{k}^1 \overline{\ell}^0 \overline{m}^1$ of $P_i$ determines the internal vertices of the local tree inside $\overline{\ell}^0$. Indeed, the two cross-edges incident with $\overline{\ell}^0$ have extremities labeled $k$ and $m$, and these vertices must therefore be internal in $T_i(\overline{\ell}^0)$.

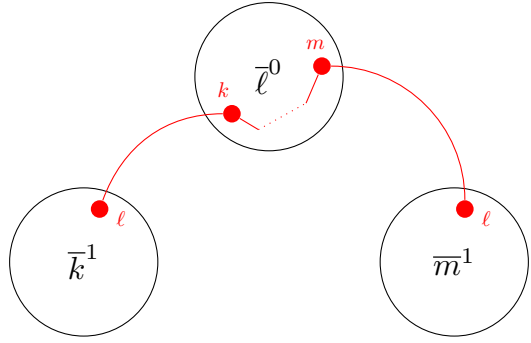
\begin{figure}[h]
    \centering
    \begin{tikzpicture}[scale=0.7]
\node[shape=circle,scale=1.2] (1) at (0,1) {$\overline{\ell}^0$};
\node[shape=circle,scale=1.2] (4) at (-3.5,-2.5) {$\overline{k}^1$};
\node[shape=circle,scale=1.2] (6) at (3.5,-2.5) {$\overline{m}^1$};

\node[shape=circle,label={[scale=0.8,text=red,xshift=-4pt,yshift=0pt]above:$k$},fill=red,scale=0.7] (2) at (-0.7,0.3) {};

\node[shape=circle,label={[scale=0.8,text=red,xshift=-3pt,yshift=0pt]above:$m$},fill=red,scale=0.7] (3) at (1,1.2) {};

\node[shape=circle,label={[scale=0.8,text=red,xshift=10pt,yshift=7pt]below:$\ell$},fill=red,scale=0.7] (5) at (-3.2,-1.5) {};

\node[shape=circle,label={[scale=0.8,text=red,xshift=10pt,yshift=7pt]below:$\ell$},fill=red,scale=0.7] (7) at (3.7,-1.5) {};

\draw (0,1) ellipse (1.4cm and 1.4cm);
\draw (-3.5,-2.5) ellipse (1.4cm and 1.4cm);
\draw (3.5,-2.5) ellipse (1.4cm and 1.4cm);
\draw[-,red] (2) to[out=180, in=70] (5);
\draw[-,red] (3) to[out=0, in=90] (7);

    \path[-,red] (2) edge (-0.2,0);
    \path[dotted,red] (-0.2,0) edge (0.7,0.5);
    \path[-,red] (3) edge (0.7,0.5);
    \end{tikzpicture}
        \caption{Subpath of $P_i$ induced by the clusters $\overline{k}^1$, $\overline{\ell}^0$, and $\overline{m}^1$}
        \label{Pi}
\end{figure}

Since each cluster is isomorphic to $Q_{n-1}$ and $n\geq 5$, by Lemma \ref{lemme1}, two CIST can be constructed within each $C$, denoted $T_1(C)$ and $T_2(C)$ with the required internal vertices. Let $T_i$ be the union of all local trees $T_i(C)$ with the cross-edges of the path $P_i$, for $i=1,2$. Then $T_i$ is connected and acyclic; hence, it is a spanning tree of $F_n$. Moreover, $T_1$ and $T_2$ are edge-disjoint, as $P_1$ and $P_2$ are edge-disjoint. Finally, by construction, no vertex is internal in both $T_1$ and $T_2$. Hence, $T_1$ and $T_2$ are two completely independent spanning trees of $F_n$.
\end{proof}

\section{Optimized algorithm for two CIST in $F_n$}
\label{sec4}
The cluster-based construction of CIST presented in the previous section proves the existence of two CIST in $F_n$ for every $n \geq 5$. However, this construction relies on recursively constructing CIST within the hypercube clusters $Q_{n-1}$ and imposes no constraints on the diameters of the resulting CIST. The purpose of this section is to obtain two CIST in $F_{n+1}$ with small diameters by starting from two CIST already constructed in $F_n$, and connecting the four copies of $F_n$ through appropriate cross-edges. To that end, a constructive algorithm called \textsc{F-2CIST} is presented, it follows the cluster-based method developed in the previous section. Also, it uses the ILP formulation in Appendix \ref{appB}, with an additional optimization phase to select inter-cluster connections that give better diameters. The algorithm is applied to $F_5$, which will serve as the base case for the recursive construction of two CIST in $F_n$ for $n \geq 6$; and is divided into three phases: \textsc{F-2CIST-Base} (Algorithm~\ref{F-2CIST}) constructs two CIST in the base case $F_5$, \textsc{F-2CIST-Enhanced} (Algorithm~\ref{F-2CIST-Enhanced}) aims to lower the diameter of the two CIST constructed in $F_5$, while \textsc{F-2CIST-Recursive} (Algorithm~\ref{F-2CIST-R}) describes the recursive extension from $F_{n}$ to $F_{n+1}$, for every $n\geq 5$.

\subsection{Construction of two CIST in $F_5$}
\label{subsecf5}
The \textsc{F-2CIST-Base} algorithm is first used to construct two CIST in $F_5$. The construction starts with the inter-cluster trees $P_1$ and $P_2$ as shown in Figure~\ref{p1p2}, where circles represent class $0$ clusters, and squares represent class $1$ clusters.

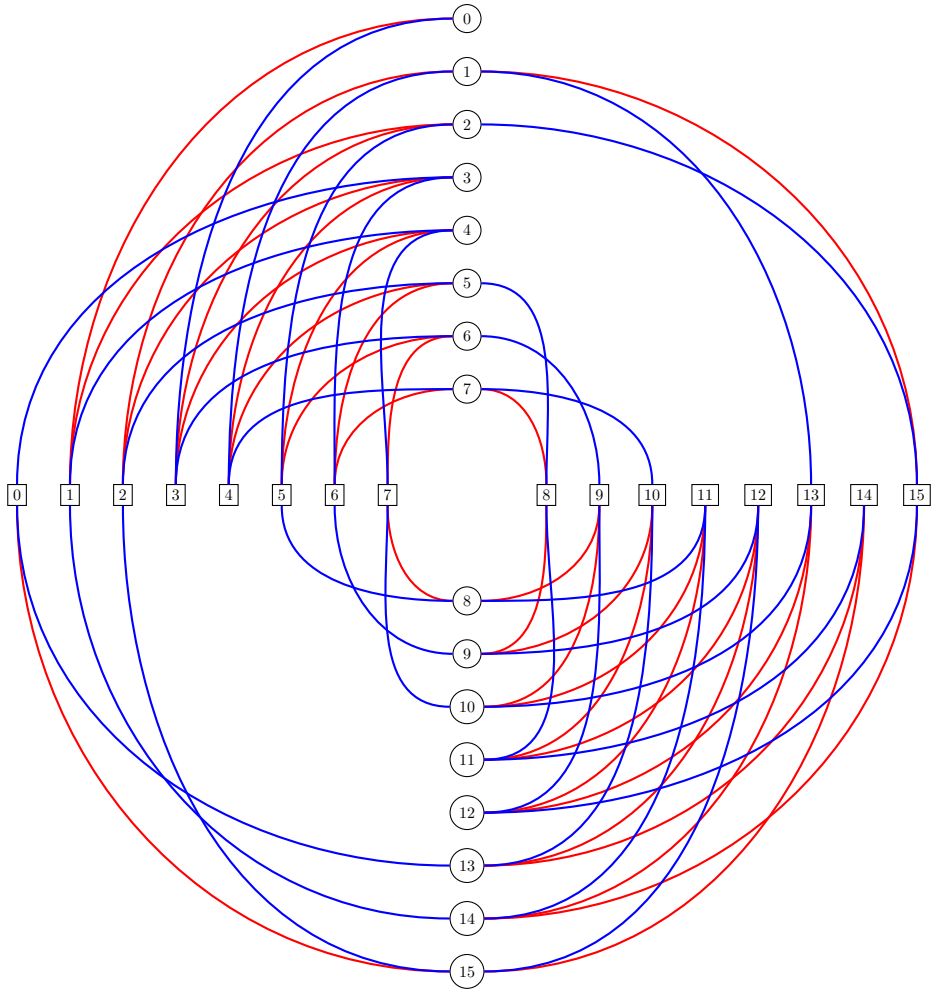
\begin{figure}[h]
    \centering
\begin{tikzpicture}[scale=0.7]
\node[shape=circle,draw=black,scale=0.6] (01) at (0,9) {$0$};
\node[shape=circle,draw=black,scale=0.6] (02) at (0,8) {$1$};
\node[shape=circle,draw=black,scale=0.6] (03) at (0,7) {$2$};
\node[shape=circle,draw=black,scale=0.6] (04) at (0,6) {$3$};
\node[shape=circle,draw=black,scale=0.6] (05) at (0,5) {$4$};
\node[shape=circle,draw=black,scale=0.6] (06) at (0,4) {$5$};
\node[shape=circle,draw=black,scale=0.6] (07) at (0,3) {$6$};
\node[shape=circle,draw=black,scale=0.6] (08) at (0,2) {$7$};
\node[shape=circle,draw=black,scale=0.6] (09) at (0,-2) {$8$};
\node[shape=circle,draw=black,scale=0.6] (010) at (0,-3) {$9$};
\node[shape=circle,draw=black,scale=0.6] (011) at (0,-4) {$10$};
\node[shape=circle,draw=black,scale=0.6] (012) at (0,-5) {$11$};
\node[shape=circle,draw=black,scale=0.6] (013) at (0,-6) {$12$};
\node[shape=circle,draw=black,scale=0.6] (014) at (0,-7) {$13$};
\node[shape=circle,draw=black,scale=0.6] (015) at (0,-8) {$14$};
\node[shape=circle,draw=black,scale=0.6] (016) at (0,-9) {$15$};

\node[shape=rectangle,draw=black,scale=0.6] (101) at (-8.5,0) {$0$};
\node[shape=rectangle,draw=black,scale=0.6] (102) at (-7.5,0) {$1$};
\node[shape=rectangle,draw=black,scale=0.6] (103) at (-6.5,0) {$2$};
\node[shape=rectangle,draw=black,scale=0.6] (104) at (-5.5,0) {$3$};
\node[shape=rectangle,draw=black,scale=0.6] (105) at (-4.5,0) {$4$};
\node[shape=rectangle,draw=black,scale=0.6] (106) at (-3.5,0) {$5$};
\node[shape=rectangle,draw=black,scale=0.6] (107) at (-2.5,0) {$6$};
\node[shape=rectangle,draw=black,scale=0.6] (108) at (-1.5,0) {$7$};
\node[shape=rectangle,draw=black,scale=0.6] (109) at (1.5,0) {$8$};
\node[shape=rectangle,draw=black,scale=0.6] (110) at (2.5,0) {$9$};
\node[shape=rectangle,draw=black,scale=0.6] (111) at (3.5,0) {$10$};
\node[shape=rectangle,draw=black,scale=0.6] (112) at (4.5,0) {$11$};
\node[shape=rectangle,draw=black,scale=0.6] (113) at (5.5,0) {$12$};
\node[shape=rectangle,draw=black,scale=0.6] (114) at (6.5,0) {$13$};
\node[shape=rectangle,draw=black,scale=0.6] (115) at (7.5,0) {$14$};
\node[shape=rectangle,draw=black,scale=0.6] (116) at (8.5,0) {$15$};

\draw[thick,red] (01) to[out=180, in=90] node[pos=0.4, black,scale=0.7] {} (102);
\draw[thick,red] (03) to[out=180, in=90] node[pos=0.4, black,scale=0.7] {} (102);
\draw[thick,red] (03) to[out=180, in=90] node[pos=0.4, black,scale=0.7] {} (104);
\draw[thick,red] (05) to[out=180, in=90] node[pos=0.4, black,scale=0.7] {} (104);
\draw[thick,red] (05) to[out=180, in=90] node[pos=0.4, black,scale=0.7] {} (106);
\draw[thick,red] (07) to[out=180, in=90] node[pos=0.4, black,scale=0.7] {} (106);
\draw[thick,red] (07) to[out=180, in=90] node[pos=0.4, black,scale=0.7] {} (108);
\draw[thick,red] (09) to[out=180, in=-90] node[pos=0.4, black,scale=0.7] {} (108);
\draw[thick,red] (09) to[out=0, in=-90] node[pos=0.4, black,scale=0.7] {} (110);
\draw[thick,red] (011) to[out=0, in=-90] node[pos=0.4, black,scale=0.7] {} (110);
\draw[thick,red] (011) to[out=0, in=-90] node[pos=0.4, black,scale=0.7] {} (112);
\draw[thick,red] (013) to[out=0, in=-90] node[pos=0.4, black,scale=0.7] {} (112);
\draw[thick,red] (013) to[out=0, in=-90] node[pos=0.4, black,scale=0.7] {} (114);
\draw[thick,red] (015) to[out=0, in=-90] node[pos=0.4, black,scale=0.7] {} (114);
\draw[thick,red] (015) to[out=0, in=-90] node[pos=0.4, black,scale=0.7] {} (116);
\draw[thick,red] (02) to[out=180, in=90] node[pos=0.4, black,scale=0.7] {} (103);
\draw[thick,red] (04) to[out=180, in=90] node[pos=0.4, black,scale=0.7] {} (103);
\draw[thick,red] (04) to[out=180, in=90] node[pos=0.4, black,scale=0.7] {} (105);
\draw[thick,red] (06) to[out=180, in=90] node[pos=0.4, black,scale=0.7] {} (105);
\draw[thick,red] (06) to[out=180, in=90] node[pos=0.4, black,scale=0.7] {} (107);
\draw[thick,red] (08) to[out=180, in=90] node[pos=0.4, black,scale=0.7] {} (107);
\draw[thick,red] (08) to[out=0, in=90] node[pos=0.4, black,scale=0.7] {} (109);
\draw[thick,red] (010) to[out=0, in=-90] node[pos=0.4, black,scale=0.7] {} (109);
\draw[thick,red] (010) to[out=0, in=-90] node[pos=0.4, black,scale=0.7] {} (111);
\draw[thick,red] (012) to[out=0, in=-90] node[pos=0.4, black,scale=0.7] {} (111);
\draw[thick,red] (012) to[out=0, in=-90] node[pos=0.4, black,scale=0.7] {} (113);
\draw[thick,red] (014) to[out=0, in=-90] node[pos=0.4, black,scale=0.7] {} (113);
\draw[thick,red] (014) to[out=0, in=-90] node[pos=0.4, black,scale=0.7] {} (115);
\draw[thick,red] (016) to[out=0, in=-90] node[pos=0.4, black,scale=0.7] {} (115);
\draw[thick,red] (02) to[out=0, in=90] node[pos=0.4, black,scale=0.7] {} (116);
\draw[thick,red] (101) to[out=-90, in=180] node[pos=0.4, black,scale=0.7] {} (016);

\draw[thick,blue] (01) to[out=180, in=90] node[pos=0.4, black,scale=0.7] {} (104);
\draw[thick,blue] (07) to[out=180, in=90] node[pos=0.4, black,scale=0.7] {} (104);
\draw[thick,blue] (07) to[out=0, in=90] node[pos=0.4, black,scale=0.7] {} (110);
\draw[thick,blue] (013) to[out=0, in=-90] node[pos=0.4, black,scale=0.7] {} (110);
\draw[thick,blue] (013) to[out=0, in=-90] node[pos=0.4, black,scale=0.7] {} (116);
\draw[thick,blue] (03) to[out=0, in=90] node[pos=0.4, black,scale=0.7] {} (116);
\draw[thick,blue] (03) to[out=180, in=90] node[pos=0.4, black,scale=0.7] {} (106);
\draw[thick,blue] (09) to[out=180, in=-90] node[pos=0.4, black,scale=0.7] {} (106);
\draw[thick,blue] (09) to[out=0, in=-90] node[pos=0.4, black,scale=0.7] {} (112);
\draw[thick,blue] (015) to[out=0, in=-90] node[pos=0.4, black,scale=0.7] {} (112);
\draw[thick,blue] (015) to[out=180, in=-90] node[pos=0.4, black,scale=0.7] {} (102);
\draw[thick,blue] (05) to[out=180, in=90] node[pos=0.4, black,scale=0.7] {} (102);
\draw[thick,blue] (05) to[out=180, in=90] node[pos=0.4, black,scale=0.7] {} (108);
\draw[thick,blue] (011) to[out=180, in=-90] node[pos=0.4, black,scale=0.7] {} (108);
\draw[thick,blue] (011) to[out=0, in=-90] node[pos=0.4, black,scale=0.7] {} (114);
\draw[thick,blue] (02) to[out=0, in=90] node[pos=0.4, black,scale=0.7] {} (114);
\draw[thick,blue] (02) to[out=180, in=90] node[pos=0.4, black,scale=0.7] {} (105);
\draw[thick,blue] (08) to[out=180, in=90] node[pos=0.4, black,scale=0.7] {} (105);
\draw[thick,blue] (08) to[out=0, in=90] node[pos=0.4, black,scale=0.7] {} (111);
\draw[thick,blue] (014) to[out=0, in=-90] node[pos=0.4, black,scale=0.7] {} (111);
\draw[thick,blue] (014) to[out=180, in=-90] node[pos=0.4, black,scale=0.7] {} (101);
\draw[thick,blue] (04) to[out=180, in=90] node[pos=0.4, black,scale=0.7] {} (101);
\draw[thick,blue] (04) to[out=180, in=90] node[pos=0.4, black,scale=0.7] {} (107);
\draw[thick,blue] (010) to[out=180, in=-90] node[pos=0.4, black,scale=0.7] {} (107);
\draw[thick,blue] (010) to[out=0, in=-90] node[pos=0.4, black,scale=0.7] {} (113);
\draw[thick,blue] (016) to[out=0, in=-90] node[pos=0.4, black,scale=0.7] {} (113);
\draw[thick,blue] (016) to[out=180, in=-90] node[pos=0.4, black,scale=0.7] {} (103);
\draw[thick,blue] (06) to[out=180, in=90] node[pos=0.4, black,scale=0.7] {} (103);
\draw[thick,blue] (06) to[out=0, in=90] node[pos=0.4, black,scale=0.7] {} (109);
\draw[thick,blue] (012) to[out=0, in=-90] node[pos=0.4, black,scale=0.7] {} (109);
\draw[thick,blue] (012) to[out=0, in=-90] node[pos=0.4, black,scale=0.7] {} (115);

\end{tikzpicture}
    \caption{Two edge-disjoint hamiltonian paths ($P_1$ in red and $P_2$ in blue) of $F_5$}
    \label{p1p2}
\end{figure}

To construct the local CIST inside each cluster, we use the ILP formulation presented in Appendix \ref{appB}. Moreover, in our construction, we additionally impose the generator constraints described in the proof of Lemma \ref{lemme1} to force the required vertices to be internal in the corresponding local trees. We recall that a tree may have either one center or two adjacent centers, depending on the parity of its diameter (see Section~\ref{sec1}). However, in what follows, we select only one of the two centers. This is sufficient for our construction,
since a single reference vertex is enough to define the parent-vector encoding defined in the next paragraph, and to root the corresponding local tree. 

Solving the model for each cluster gives the local CIST listed in Appendix \ref{appA}, where the trees of each cluster are represented using a parent-vector encoding. More precisely, for a cluster with $N$ vertices, a tree is encoded by a vector $P=(p_1,p_2,\dots,p_N)$, where the $i^{th}$ coordinate $p_i$ denotes the parent of vertex $i$ in the tree. Here, the parent of a vertex is the next vertex on the path from this vertex to the center (or root) of the tree. The center is defined to be its own parent; that is, if the vertex $i$ is the center, then $p_i=i$.
\begin{algorithm}
\caption{\textsc{F-2CIST-Base}}
\label{F-2CIST}
\KwIn{The dual-cube $F_5$}
\KwOut{Two completely independent spanning trees of $F_5$}

    Construct $\mathrm{KF}_5$\;

    Construct $P_1 (p = 1)$ and $P_2 (q = 3)$ in $\mathrm{KF}_5$ (Lemma \ref{lemme2}) \;

\ForEach{cluster $C$ of $F_5$}{

    Let $S_1(C)$ be the set of vertices of $C$ incident with the cross-edges of $P_1$ \;

    Let $S_2(C)$ be the set of vertices of $C$ incident with the cross-edges of $P_2$ \;

    \If{$|S_1(C)|=1$}{
        Add to $S_1(C)$ one vertex of $C\setminus (S_1(C)\cup S_2(C))$ \;
    }

    \If{$|S_2(C)|=1$}{
        Add to $S_2(C)$ one vertex of $C\setminus (S_1(C)\cup S_2(C))$\;
    }

    Construct two CIST $T_1(C)$ and $T_2(C)$, such that the vertices of $S_1(C)$ are generators of $T_1(C)$, and the vertices of $S_2(C)$ are generators of $T_2(C)$ (Lemma \ref{lemme1})\;
}
    \Return $T_1,T_2$ \;
\end{algorithm}

\subsection{Optimization of the inter-cluster connections}
The local CIST within the clusters are compatible with the inter-cluster connections, and can therefore be assembled into two global CIST $T_1$ and $T_2$. However, since the trees connecting the clusters are paths, this construction does not necessarily yield trees with small diameters. In this regard, \textsc{F-2CIST-Enhanced} is used to improve the diameters of $T_1$ and $T_2$. 

For each $i\in\{1,2\}$, a weighted graph $G_i$ is defined on the clusters of $F_n$, with $V(G_i)=V(\mathrm{KF}_n)$.
Two clusters $C_u$ and $C_v$ are adjacent in $G_i$ whenever there exists a cross-edge $(u,v)\in E(F_n)$ such that $u\in C_u$, $v\in C_v$, and $u$ and $v$ are internal vertices of the corresponding local trees $T_i(C_u)$ and $T_i(C_v)$. If $c_u$ and $c_v$ denote the centers of these local trees, the distance associated with this connection in $G_i$ is
\[
d_i(C_u,C_v)
=
d_{T_i(C_u)}(c_u,u)
+1+
d_{T_i(C_v)}(v,c_v),
\]
where the term $1$ corresponds to the cross-edge $(u,v)$. Then, \textsc{F-2CIST-Enhanced} searches for a shortest path tree of $G_i$ with minimum weighted eccentricity. Consider each cluster $C$ as a possible center of the tree. Dijkstra's algorithm is used to compute the shortest distances from $C$ to all other vertices, and
\[
e_i(C)=\max_{C'\in V(G_i)} d_{G_i}(C,C')
\]
is recorded. A cluster $C_i^\star$ minimizing this expression is selected, and the shortest-path tree rooted at $C_i^\star$ determines the inter-cluster connections kept for the final construction of $T_i$. Then, the corresponding cross-edges are combined with the local trees of all clusters. Finally, the diameter of each resulting tree is computed by the two-BFS procedure: a first BFS identifies a farthest vertex $x$, and a second BFS from $x$ identifies a farthest vertex $y$, giving $\operatorname{diam}(T_i)=d_{T_i}(x,y)$.

The resulting inter-cluster trees for $T_1$ and $T_2$ are represented in Figures~\ref{g1} and~\ref{g2}, respectively, in Appendix~\ref{appC}. Cluster $\overline{7}^1$ is selected as the center of $T_1$; however, the diameter that is computed in $G_1$ measures only the maximum distance between the selected centers of the local trees. In this case, the two extreme centers belong to clusters $\overline{3}^0$ and $\overline{11}^1$. Hence, this value does not yet represent the diameter of the global tree $T_1$ because it does not include the additional distances from these local centers to the actual farthest vertices inside their respective clusters. We thus refer to this intermediate value as {\em clusteral diameter}, denoted by:
\[
\operatorname{diam}_C(T_1) = d_{T_1}(\overline{3}^0,\overline{11}^1) = 21
\]
Therefore, the usual diameter $\operatorname{diam}(T_1)$ is equal to:
\[
\operatorname{diam}_C(T_1) + d_{\overline{3}^0}(15,12) + d_{\overline{11}^1}(13,8) = 29\]
Similarly, the center of $T_2$ is cluster $\overline{5}^0$. Its clusteral diameter and usual diameters are:
\[
\operatorname{diam}_C(T_2) = d_{T_2}(\overline{4}^1,\overline{6}^1) = 23
\]
\[\operatorname{diam}_C(T_2) + d_{\overline{4}^1}(11,14) + d_{\overline{6}^1}(10,0) = 31
\]
Figure~\ref{clusteral} illustrates the clusteral diametral chain of $T_1$, that is, the path connecting the two extremal clusters $3^0$ and $11^1$ in $T_1$. It also shows the complete diametral chain after adding the local paths from the centers of the two clusters previously cited to their corresponding farthest vertices inside them.
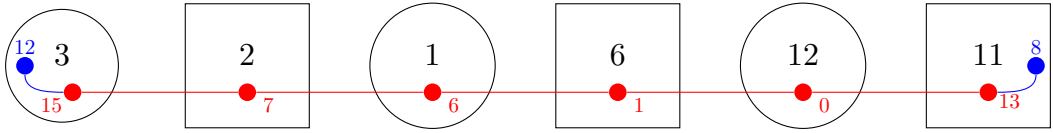
\begin{figure}[h]
    \centering
    \begin{tikzpicture}[scale=0.7]
\node[shape=rectangle,scale=7,draw=black] (1) at (0,0) {};
\node[shape=rectangle,scale=1.2] (11) at (0,0.2) {$6$};

\node[shape=circle,scale=5,draw=black] (3) at (-3.5,0) {};
\node[shape=circle,scale=1.2] (33) at (-3.5,0.2) {$1$};

\node[shape=circle,scale=5,draw=black] (6) at (3.5,0) {};
\node[shape=circle,scale=1.2] (66) at (3.5,0.2) {$12$};
    
\node[shape=rectangle,scale=7,draw=black] (7) at (-7,0) {};
\node[shape=rectangle,scale=1.2] (77) at (-7,0.2) {$2$};

\node[shape=rectangle,scale=7,draw=black] (8) at (7,0) {};
\node[shape=circle,scale=1.2] (88) at (7,0.2) {$11$};

\node[shape=circle,scale=4.5,draw=black] (10) at (-10.5,0) {};
\node[shape=circle,scale=1.2] (1010) at (-10.5,0.2) {$3$};

\node[shape=circle,label={[scale=0.8,text=red,xshift=10pt,yshift=7pt]below:$13$},fill=red,scale=0.7] (2) at (7,-0.5) {};

\node[shape=circle,label={[scale=0.8,text=red,xshift=10pt,yshift=5pt]below:$1$},fill=red,scale=0.7] (37) at (0,-0.5) {};

\node[shape=circle,label={[scale=0.8,text=red,xshift=10pt,yshift=5pt]below:$6$},fill=red,scale=0.7] (5) at (-3.5,-0.5) {};

\node[shape=circle,label={[scale=0.8,text=red,xshift=10pt,yshift=5pt]below:$0$},fill=red,scale=0.7] (7) at (3.5,-0.5) {};

\node[shape=circle,label={[scale=0.8,text=red,xshift=10pt,yshift=5pt]below:$7$},fill=red,scale=0.7] (9) at (-7,-0.5) {};

\node[shape=circle,label={[scale=0.8,text=red,xshift=-10pt,yshift=5pt]below:$15$},fill=red,scale=0.7] (11) at (-10.3,-0.5) {};

\node[shape=circle,label={[scale=0.8,text=blue,xshift=0pt,yshift=20pt]below:$12$},fill=blue,scale=0.7] (12) at (-11.2,0) {};

\node[shape=circle,label={[scale=0.8,text=blue,xshift=0pt,yshift=20pt]below:$8$},fill=blue,scale=0.7] (13) at (7.9,0) {};

    \path[-,red] (37) edge (7);
    \path[-,red] (2) edge  (7);
    \path[-,red] (37) edge (5);
    \path[-,red] (9) edge (5);
    \path[-,red] (9) edge (11);

\draw[-,blue] (11) to[out=180, in=-90] (12);
\draw[-,blue] (2) to[out=0, in=-90] (13);

    \end{tikzpicture}
        \caption{Clusteral and complete diametral chains of $T_1$}
        \label{clusteral}
\end{figure}

\begin{algorithm}
\caption{\textsc{F-2CIST-Enhanced}}
\label{F-2CIST-Enhanced}
\KwIn{Two completely independent spanning trees $T'_1$ and $T'_2$ of $F_5$.}
\KwOut{Two CIST $T_1$ and $T_2$ with smaller diameters}

\For{$i=1,2$}{

    Construct the weighted graph $G_i$ \;

    \ForEach{pair of distinct clusters $C,C'$}{

        \If{$ \exists (x,y) \in E(F_5)$ such that $x\in V(C)$, $y\in V(C')$, and $x,y$ are internal vertices of $T'_i(C)$ and $T'_i(C')$, respectively}{
            Add $(C,C')$ to $E_i$ \;
            Let $c_i(C)$ be the center of $T'_i(C)$ \;
            Set $d_i(C,C')= d_{T'_i(C)}(c_i(C),x)+1+d_{T'_i(C')}(y,c_i(C'))$ \;
        }
    }

    \ForEach{cluster $C$}{

        Apply Dijkstra's algorithm from cluster $C$ in the weighted graph $G_i$ \;

        Store the eccentricity
        \[
        e_i(C)=\max_{C'\ne C} d_{G_i}(C,C') \;
        \]
    }

    Choose a cluster $C_i^\star$ minimizing $r_i(C)$.

    Let $T_i$ be the shortest-path tree induced by $C_i^\star$ \;

    \ForEach{edge $(C,C') \in E(T_i)$}{
    Select a cross-edge $(x,y)$ reaching  $d_i( (C,C') )$ \;}

    Construct $T_i$ as the union of all trees $T'_i(C)$ and the selected cross-edges\;
}

    \Return $T_1,T_2$ \;
\end{algorithm}

Therefore, combining the algorithms \textsc{F-2CIST-Base} and \textsc{F-2CIST-Enhanced}, we have the following result.

\begin{thm}
    Together, Algorithms \textsc{F-2CIST-Base} and \textsc{F-2CIST-Enhanced} construct two completely independent spanning trees $T_1$ and $T_2$ in $F_5$ with diameters $29$ and $31$, respectively.
\end{thm}

\subsection{Construction of two CIST in $F_n$ for $n \geq 6$}
The algorithm \textsc{F-2CIST-Recursive} provides a recursive structure for constructing two CIST in the dual-cube. More precisely, starting from the isomorphic pairs of CIST in the four subcopies of $F_{n-1}$, only three cross-edges are needed to connect these subcopies and obtain a global spanning tree. However, the choice of these cross-edges is crucial to avoid a significant increase in diameter. For this reason, the cross-edges are selected to connect the centers of the local trees whenever possible. If such connections are not available, we choose cross-edges whose extremities are as close as possible to these centers. In what follows, this recursive construction is shown to be possible for $F_6$ and, by extension, for $F_n$ with $n \geq 6$ using the copies of $F_5$. 

\begin{lemme}
Let $n\geq 6$. If each copy $F_{n-1}^{b_2b_1}$ with $b_1,b_2 \in \{0,1\}$ in the recursive construction of $F_n$ contains two isomorphic completely independent spanning trees, then $F_n$ admits two completely independent spanning trees.
\label{lemme4}
\end{lemme}

\begin{proof}
Let $F_n$ be constructed from the four copies $F_{n-1}^{00}, F_{n-1}^{01}, F_{n-1}^{10}, F_{n-1}^{11}$. The goal is to construct two completely independent spanning trees $T_1$ and $T_2$ in $F_n$. For each copy $F_{n-1}^{b_2b_1}$ with $b_1,b_2 \in \{0,1\}$, let $T_1'$ and $T_2'$ be a pair of completely independent spanning trees; these pairs of trees are isomorphic for all four copies. According to the recursive construction of the dual-cube, these proposed three cross-edges between the following pairs of copies form a tree on the four copies and can then be used for the first tree $T_1$:
\[
e_1^{1} \text{ joins } F_{n-1}^{00} \text{ to } F_{n-1}^{01}
\]
\[
e_1^{2} \text{ joins } F_{n-1}^{00} \text{ to } F_{n-1}^{10}
\]
\[
e_1^{3} \text{ joins } F_{n-1}^{10} \text{ to } F_{n-1}^{11}
\]
Since the trees $T_1'$ of each copy are isomorphic to each other, we choose the edges $e_1^1=(x_1,y_1)$, $e_1^2 = (x_2,y_2)$, and $e_1^3 = (x_3,y_3)$ with:
\[
\begin{aligned}
x_1&=(0,0,x_{n-2},\ldots,x_1,0,y_{n-2},\ldots,y_1),\\
y_1&=(0,0,x_{n-2},\ldots,x_1,1,y_{n-2},\ldots,y_1);
\end{aligned}
\]
\[
\begin{aligned}
x_2&=(1,0,x_{n-2},\ldots,x_1,0,y_{n-2},\ldots,y_1),\\
y_2&=(1,1,x_{n-2},\ldots,x_1,0,y_{n-2},\ldots,y_1);
\end{aligned}
\]
\[
\begin{aligned}
x_3&=(0,1,x_{n-2},\ldots,x_1,0,y_{n-2},\ldots,y_1),\\
y_3&=(0,1,x_{n-2},\ldots,x_1,1,y_{n-2},\ldots,y_1).
\end{aligned}
\]

Similarly, for the second tree $T_2$, we choose three distinct cross-edges $e_2^{1}, e_2^{2}, e_2^{3}$ that join the trees $T_2'$ using the same construction as $T_1$. Then, we define for $i\in\{1,2\}$:
\[
T_i = T_i^{00}\cup T_i^{01}\cup T_i^{10}\cup T_i^{11} \cup \{e_i^{1},e_i^{2},e_i^{3}\}
\]
An easy way to see that each $T_i$ is connected and acyclic is to contract each copy $F_{n-1}^{b_2b_1}$ into a single vertex. We then have a path on four vertices. Therefore, $T_1$ and $T_2$ are two completely independent spanning trees of $F_n$.

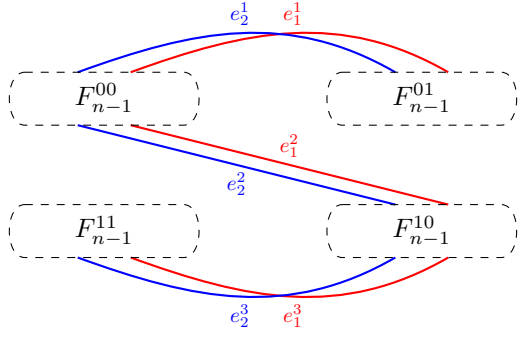
\begin{figure}[h]
    \centering
\begin{tikzpicture}[scale = 0.7]
    \node[shape=circle] (1) at (1,0) {};
    \node[shape=circle] (12) at (2,0) {$F_{n-1}^{00}$};    
    \node[shape=circle] (2) at (3,0) {};
    \node[shape=circle] (3) at (7,0) {};
    \node[shape=circle] (34) at (8,0) {$F_{n-1}^{01}$};   
    \node[shape=circle] (4) at (9,0) {};
    \node[shape=circle] (5) at (1,-2.5) {};
    \node[shape=circle] (56) at (2,-2.5) {$F_{n-1}^{11}$};    
    \node[shape=circle] (6) at (3,-2.5) {};
    \node[shape=circle] (7) at (7,-2.5) {};
    \node[shape=circle] (78) at (8,-2.5) {$F_{n-1}^{10}$};
    \node[shape=circle] (8) at (9,-2.5) {};

    \path [-,red,thick] (2.5,-0.5) edge node[midway,above,scale=0.8,text=red] {$e_1^2$} (8.5,-2);
    \path [-,blue,thick] (1.5,-0.5) edge node[midway,below,scale=0.8,text=blue] {$e_2^2$} (7.5,-2);

    \draw[-,red,thick] (2.5,0.5) to[out=20, in= 150] node[midway,above,scale=0.8,text=red] {$e_1^1$}
    (8.5,0.5);
    \draw[-,blue,thick] (1.5,0.5) to[out=20, in= 150] node[midway,above,scale=0.8,text=blue] {$e_2^1$} (7.5,0.5);

    \draw[-,red,thick] (2.5,-3) to[out=-20, in= -150] node[midway,below,scale=0.8,text=red] {$e_1^3$} (8.5,-3);
    \draw[-,blue,thick] (1.5,-3) to[out=-20, in= -150] node[midway,below,scale=0.8,text=blue] {$e_2^3$}(7.5,-3);

    \begin{scope}[fill opacity=0,dashed]
    \filldraw[] ($(1)+(-0.5,0.5)$)
        to[out=0,in=-180] ($(2) + (0.5,0.5)$)    
        to[out=0,in=0] ($(2) + (0.5,-0.5)$)
        to[out=180,in=0] ($(1) + (-0.5,-0.5)$)
        to[out=180,in=180] ($(1) + (-0.5,0.5)$) ;
    \end{scope}    

    \begin{scope}[fill opacity=0,dashed]
    \filldraw[] ($(3)+(-0.5,0.5)$)
        to[out=0,in=-180] ($(4) + (0.5,0.5)$)    
        to[out=0,in=0] ($(4) + (0.5,-0.5)$)
        to[out=180,in=0] ($(3) + (-0.5,-0.5)$)
        to[out=180,in=180] ($(3) + (-0.5,0.5)$) ;
    \end{scope}    

    \begin{scope}[fill opacity=0,dashed]
    \filldraw[] ($(5)+(-0.5,0.5)$)
        to[out=0,in=-180] ($(6) + (0.5,0.5)$)    
        to[out=0,in=0] ($(6) + (0.5,-0.5)$)
        to[out=180,in=0] ($(5) + (-0.5,-0.5)$)
        to[out=180,in=180] ($(5) + (-0.5,0.5)$) ;
    \end{scope}    

    \begin{scope}[fill opacity=0,dashed]
    \filldraw[] ($(7)+(-0.5,0.5)$)
        to[out=0,in=-180] ($(8) + (0.5,0.5)$)    
        to[out=0,in=0] ($(8) + (0.5,-0.5)$)
        to[out=180,in=0] ($(7) + (-0.5,-0.5)$)
        to[out=180,in=180] ($(7) + (-0.5,0.5)$) ;
    \end{scope}    
\end{tikzpicture}
        \caption{Recursive construction of $F_n$ and cross-edges used for the two CIST}
        \label{fig7}

\end{figure}
\end{proof}
Therefore, from Lemmas~\ref{thm0} and \ref{lemme4}, we have the following result.

\begin{thm}\label{thm6}
    For every $n \geq 6$, Algorithm~\textsc{F-2CIST-Recursive} constructs two completely independent spanning trees in $F_n$ with diameters $6n-1$ and $6n+1$, respectively.
\end{thm}

\begin{proof}
We first describe the construction in $F_6$ from the two CIST obtained in $F_5$, and then show that the same connection pattern repeats at every higher dimension. Let $T_1'$ and $T_2'$ denote the two CIST of $F_5$ produced by Algorithms~\textsc{F-2CIST-Base} and \textsc{F-2CIST-Enhanced}. Their respective centers belong to clusters $17$ and $06$. The graph $F_6$ is decomposed into four copies $F_5^{00}, F_5^{01}, F_5^{10}, F_5^{11}$. In each copy, an isomorphic copy of $T_1'$ and $T_2'$ is placed. Then, for each global tree, three cross-edges are added so that the four copies are connected according to a tree.

\textbf{Construction of $T_1$:}\\
The center of $T_1'$ is in cluster $\overline{6}^1$, which is a class $1$ cluster. According to the recursive construction of the dual-cube, cross-edges that connect the vertices of class $1$ are the ones that connect $F_5^{00}$ to $F_5^{10}$, and $F_5^{01}$ to $F_5^{11}$. We therefore select
\[
e_1^1 = (S_1^{00}(\overline{6}^1), S_1^{10}(\overline{6}^1) )
\]
and
\[
e_1^2 = (S_1^{01}(\overline{6}^1), S_1^{11}(\overline{6}^1) ),
\]
where each edge joins the corresponding centers of the two copies of cluster $\overline{6}^1$. These two edges form a pair of connected copies. To join these pairs, a class $0$ cross-edge is required. In this regard, cluster $\overline{1}^0$ is chosen, because its selected center is adjacent in the local tree to the selected center of cluster $\overline{6}^1$. The third edge is therefore
\[
e_1^3 = (S_0^{00}(\overline{1}^0), S_0^{01}(\overline{1}^0) ).
\]
Thus, after contracting each copy of $F_5$ into one vertex, the three selected cross-edges form a tree on the four copies. Since their extremities are internal vertices of the corresponding local copies of $T_1'$, their union with these four copies produces a spanning tree $T_1$ of $F_6$.

With this construction, the diameter obtained for $T_1$ is the radius of the two extremal copies of $T_1'$ in $F^{11}_5$ and $F^{10}_5$ in Figure~\ref{fig7}, together with the additional distances created by the selected cross-edges. By repeating the same construction recursively, let $D_1(n)$ and $r_1(n)$ denote the respective diameter and radius of $T_1$ in $F_n$. In the recursive construction of $T_1$, a diametral path consists of two paths of length $r_1(n-1)$, together with five edges connecting the four copies; it follows that $D_1(n)=2r_1(n-1)+5$. Moreover, since $D_1(5)$ is odd, we get $D_1(n-1)=2r_1(n-1)-1$; combining the two previous equalities, we get $D_1(n)=D_1(n-1)+6$. Using $D_1(5)=29$, we obtain:
\[
\operatorname{diam}(T_1) = 6n-1 \ \mbox{ for } \ n \geq 5.
\]    

\textbf{Construction of $T_2$:}\\
The center of $T_2'$ is in cluster $\overline{5}^0$, which is a class-$0$ cluster. Hence, the recursive construction provides cross-edges between the pairs $F_5^{00}$ to $F_5^{01}$, and $F_5^{10}$ to $F_5^{11}$. We select
\[
e_2^1 = (S_0^{00}(\overline{5}^0), S_0^{01}(\overline{5}^0) )
\]
and
\[
e_2^2 = (S_0^{10}(\overline{5}^0), S_0^{11}(\overline{5}^0) ),
\]
again joining the corresponding centers. To connect the two resulting pairs of copies, a class $1$ cross-edge is required. We use cluster $\overline{12}^1$, whose selected center is adjacent in the local tree to the selected center of cluster $\overline{5}^0$. Thus, the third edge is
\[
e_2^3 = (S_1^{01}(\overline{12}^1), S_1^{11}(\overline{12}^1) ).
\]

Similar to $T_1$, these three cross-edges induce a tree on the four copies and have extremities that are internal in the corresponding local trees. Consequently, their union with the four copies of $T_2'$ produces a spanning tree $T_2$ of $F_6$.  Therefore, with this construction, the diameter obtained for $T_2$ is the radius of the two extremal copies of $T_2'$ in $F^{00}_5$ and $F^{10}_5$ in Figure~\ref{fig7}, together with the additional distances created by the selected cross-edges.
Using the same reasoning as for $T_1$ but with $D_2(5)=31$, we get the following equality:
\[
\operatorname{diam}(T_2) = 6n+1 \ \mbox{ for } \ n \geq 5
\]    

The cross-edges selected for $T_1$ and $T_2$ are distinct, and the local trees are pairwise edge-disjoint and internally vertex-disjoint. Thus, for every $n\geq 6$, Algorithm \textsc{F-2CIST} constructs two completely independent spanning trees in $F_n$ with diameters $6n-1$ and $6n+1$, respectively.
\end{proof}

\begin{algorithm}
\caption{\textsc{F-2CIST-Recursive}}
\label{F-2CIST-R}
\KwIn{$n \geq 6$}
\KwOut{Two completely independent spanning trees $T_1$ and $T_2$ of $F_n$}

Construct the four copies $F^{00}_{n-1}$, $F^{01}_{n-1}$, $F^{10}_{n-1}$, and $F^{11}_{n-1}$ of $F_{n-1}$ \;

\ForEach{$b_1,b_2\in\{0,1\}$}{

    \If{$n=6$}{
        Apply \textsc{F-2CIST-Base} to $F_5^{b_2b_1}$ to obtain
        two local CIST $T_1^{b_2b_1}$ and $T_2^{b_2b_1}$\;
        
        Apply \textsc{F-2CIST-Enhanced} to
        $T_1^{b_2b_1}$ and $T_2^{b_2b_1}$ \;
    }
    
    \Else{    
        Apply \textsc{F-2CIST-Recursive} to $F^{b_2b_1}_{n-1}$ to obtain two local CIST $T^{b_2b_1}_1$ and $T^{b_2b_1}_2$ \;
    }
}

\For{$i=1,2$}{
    Select three cross-edges $e_i^1,e_i^2,e_i^3$, where
    $e_i^1$ connects $F^{00}_{n-1}$ to $F^{01}_{n-1}$,
    $e_i^2$ connects $F^{00}_{n-1}$ to $F^{10}_{n-1}$, and
    $e_i^3$ connects $F^{10}_{n-1}$ to $F^{11}_{n-1}$ \;

\For{$\ell=1,2,3$}{

    Let $F^{\alpha}_{n-1}$ and $F^{\beta}_{n-1}$ be the two copies connected by $e_i^\ell$ \;

    Choose $e_i^\ell= (x,y) $, with
    $x\in V(F^{\alpha}_{n-1})$ and $y\in V(F^{\beta}_{n-1})$, among all cross-edges such that $x$ and $y$ are internal vertices of $T^{\alpha}_i$ and $T^{\beta}_i$, respectively, and such that
    \[
    d_{T^{\alpha}_i}(c^{\alpha}_i,x)+1+d_{T^{\beta}_i}(y,c^{\beta}_i)\]
    is minimum\;
}

    Construct
    \[
    T_i =T^{00}_i\cup T^{01}_i\cup T^{10}_i\cup T^{11}_i\cup \{e_i^1,e_i^2,e_i^3\}\;
    \]
}

\Return $T_1,T_2$\;
\end{algorithm}

\section{Concluding remarks}
\label{sec5}
In this paper, we studied the problem of the existence and construction of two completely independent spanning trees in the dual-cube $F_n$ (for $n \geq 5$). First, we exploited the cluster structure of $F_n$ to transform the problem into two subproblems: an inter-cluster tree construction, handled via the graph $\mathrm{KF}_n$, and local constructions within the clusters (each isomorphic to a hypercube $Q_{n-1}$), using an ILP program. Then, we proposed the F-2CIST algorithm, which provides an explicit construction starting from the base case $F_5$ and extending it recursively to higher dimensions, while selecting appropriate cross-edges to reduce the diameters of the resulting trees.

Moreover, comparing the hypercube $Q_n$ with the dual-cube $F_{\lfloor \frac{n+1}{2} \rfloor}$ having the same number of vertices $2^n$ when $n \geq 6$ (i.e., only when $n$ is odd), the best known diameter of the two CIST in $Q_n$ is $2n - 1$ \cite{8}  while the diameter of the two CIST in $F_{\lfloor \frac{n+1}{2} \rfloor}$ constructed by our method is at most 3n+4, see Theorem~\ref{thm6}. This comparison shows that although the dual-cube has a lower degree and fewer edges than the hypercube, the diameters of the CIST are relatively close. We conclude that the dual-cube provides an efficient structure for CIST construction; it significantly reduces network density while preserving a diameter comparable to that of the hypercube.

Finally, since each cluster of $F_n$ is isomorphic to $Q_{n-1}$, and since the graph $\mathrm{KF}_n$ connecting the clusters of $F_n$ admits multiple edge-disjoint connecting structures, it is natural to ask whether more CIST in $Q_n$ implies more CIST in $F_{n+1}$. We believe it is the case and propose the following conjecture:

\begin{conjecture}
    If $Q_n$ has $k$ completely independent spanning trees, then $F_{n+1}$ also has $k$ completely independent spanning trees.
\end{conjecture}

\bibliographystyle{abbrvnat}
\bibliography{reference.bib}

\appendix
\appendixpage
\addappheadtotoc

\section{Linear modeling} \label{appB}
Let $G = (V,E)$ be a graph such that $|V(G)| = N$ and $|E| = m$. Let $A = (a_{ij})$ be the adjacency matrix of $G$ such that $a_{ij} = 1$ if vertex $i$ is adjacent to vertex $j$, otherwise $a_{ij} = 0$. The goal is to find a set of $p$ completely independent spanning trees (if possible), each with minimal radius. To do that, we introduce four binary variables. For $i\in[1,p]$ and $j\in[1,N]$, $y_{ij}$ is a binary variable equal to $1$ if vertex $j$ is an internal vertex of tree $i$, and $0$ otherwise. For $i\in[1,p]$ and $j,k\in[1,N]$, $u_{ijk}$ is a binary variable equal to $1$ if edge $(j,k)$ belongs to tree $i$, and $0$ otherwise. For $i\in[1,p]$ and $j\in[1,N]$, $v_{ij}$ is a binary variable equal to $1$ if vertex $j$ is the center of tree $i$, and $0$ otherwise. For $i\in[1,p]$ and $j,k\in[1,N]$, $p_{ijk}$ is a binary variable equal to $1$ if vertex $j$ is the parent of vertex $k$ in tree $i$, and $0$ otherwise. For $i\in[1,p]$ and $j\in[1,N]$, $d_{ij}\in\mathbb{N}$ denotes the distance from vertex $j$ to the center of tree $i$. Finally, $D\in\mathbb{N}$ denotes the maximum radius among the $p$ spanning trees. Let $\overline{i,j}$ denote the set of integers from $i$ to $j$, inclusive. The ILP formulation of the problem is given as follows:

Minimize $D$ \nopagebreak[4]

subject to 
\begin{empheq}[left=\empheqlbrace]{align}
& \sum_{1\leq j<k\leq N}^N u_{ijk}A_{jk}=N-1 && \forall i = \overline{1,p} \label{eq1} \\
& u_{ijk} \leq A_{jk} && \forall i = \overline{1,p},\ j,k = \overline{1,N} \label{eq2} \\
& \sum\limits_{i = 1}^p u_{ijk} \leq 1 && \forall j,k = \overline{1,N} \label{eq3} \\
& \sum\limits_{k = 1}^N u_{ijk} \geq 1 && \forall i = \overline{1,p} ,\ j = \overline{1,N} \label{eq4}\\
& u_{ijk} = u_{ikj} && \forall i = \overline{1,p},\ j,k = \overline{1,N} \label{eq5} \\
& \sum\limits_{i = 1}^p y_{ij} \leq 1 && \forall j = \overline{1,N} \label{eq6} \\
& 2 \ y_{ij} \leq \sum\limits_{k = 1}^N u_{ijk} && \forall i = \overline{1,p},\ j = \overline{1,N} \label{eq8} \\
& \sum\limits_{k = 1}^N u_{ijk} \leq (N - 2) \ y_{ij} + 1
&& \forall i = \overline{1,p},\ \  j = \overline{1,N} \label{eq9}\\
& \sum\limits_{j = 1}^N v_{ij} = 1 && \forall i = \overline{1,p} \label{eq10}\\
& p_{ijk} \leq u_{ijk} && \forall i = \overline{1,p},\ j,k = \overline{1,N} \label{eq11}\\
& \sum\limits_{k = 1}^N p_{ikj} = 1 - v_{ij} && \forall i = \overline{1,p},\ j = \overline{1,N} \label{eq12}\\
& d_{ij} \leq M (1 - v_{ij}) && \forall i = \overline{1,p},\  j = \overline{1,N} \label{eq13}\\
& d_{ij} \geq d_{ik} + 1 - M(1 - p_{ikj}) && \forall i = \overline{1,p},\ j,k = \overline{1,N} \label{eq14}\\
& d_{ij} \leq d_{ik} + 1 + M(1 - p_{ikj}) && \forall i = \overline{1,p},\ j,k = \overline{1,N} \label{eq15}\\
& d_{ij} \leq D && \forall i = \overline{1,p},\ j = \overline{1,N} \label{eq16}\\
& y_{ij},u_{ijk},v_{ij},p_{ijk} \in \{0,1\} && \forall i = \overline{1,p},\ j,k = \overline{1,N} \\
& d_{ij},D \in \mathbb{N} && \forall i = \overline{1,p},\ j = \overline{1,N} 
\end{empheq}

\vspace{1cm}
First, Eqs.~(\ref{eq1})-(\ref{eq5}) define the edge sets of the trees. For each tree, Eq.~(\ref{eq1}) guarantees that each tree must have $N-1$ edges, while Eq.~(\ref{eq2}) ensures that only edges of $G$ can be selected. Eq.~(\ref{eq3}) ensures that each edge can be used by at most one tree, Eq.~(\ref{eq4}) ensures that each vertex must be incident to an edge that belongs to a tree, and this for each tree, \textit{i.e.} each tree spans $G$, and Eq.~(\ref{eq5}) comes from the fact that the graph $G$ is not oriented. The internal vertices are then characterized by Eqs.~(\ref{eq8}) and~(\ref{eq9}), which ensure that each internal vertex must have a degree of at least two in the tree it is internal to. Combined with Eq.~(\ref{eq6}), this ensures that no vertex is internal to more than one tree. Next, Eq.~(\ref{eq10}) determines the center of each tree. Eqs.~(\ref{eq11}) and~(\ref{eq12}) establish a parent relation between the vertices of each tree; this is equivalent to the flow method and thus guarantees the connectivity of each tree. Eqs.~(\ref{eq13})-(\ref{eq15}) compute the distance between each vertex and the center of each tree; the maximal distance is the radius. Finally, Eq.~(\ref{eq16}) is used to ensure that the objective is the best radius possible between all trees.

\section{Local CIST obtained for the clusters of $F_5$} \label{appA}
Since the dual-cube $F_5$ is composed of $32$ clusters, each isomorphic to the hypercube $Q_4$, Tables~\ref{local1} and~\ref{local2} present the local trees constructed within these clusters to obtain the two global CIST, $T_1$ and $T_2$, using the notation introduced in Subsection~\ref{subsecf5}.

\begin{center}
\begin{longtable}{|c|c|c|c|}
\hline
Cluster $C$ & Generators & Center & Tree $T_1(C)$ \\
\hline
$\overline{0}^0$ & $1$ & $11$ & $(8,9,6,1,6,1,7,15,9,11,11,11,8,12,15,11)$\\
\hline
$\overline{1}^0$ & $2,15$ & $6$ & $(1,3,6,2,6,1,6,3,12,8,8,15,14,15,6,14)$\\
\hline
$\overline{2}^0$ & $1,3$ & $7$ & $(1,3,6,7,0,7,7,7,12,1,14,3,14,12,6,11)$\\
\hline
$\overline{3}^0$ & $2,4$ & $15$ & $(2,3,10,2,12,4,4,5,10,13,14,3,14,12,14,14)$\\
\hline
$\overline{4}^0$ & $3,5$ & $14$ & $(2,3,10,2,12,4,4,5,10,13,14,3,14,12,14,14)$ \\
\hline
$\overline{5}^0$ & $4,6$ & $1$ & $(1,1,4,1,5,1,4,3,10,13,11,3,4,13,6,11)$\\
\hline
$\overline{6}^0$ & $5,7$ & $1$ & $(1,1,0,1,5,1,7,5,0,8,2,15,8,12,12,7)$\\
\hline
$\overline{7}^0$ & $6,8$ & $15$ & $(2,5,6,2,6,13,14,5,12,8,8,15,13,15,15)$\\
\hline
$\overline{8}^0$ & $7,9$ & $3$ & $(1,3,0,3,6,7,7,3,0,1,14,9,14,9,6,13)$\\
\hline
$\overline{9}^0$ & $8,10$ & $10$ & $(8,9,10,7,0,7,4,15,10,8,10,15,4,9,10,14)$\\
\hline
$\overline{10}^0$ & $9,11$ & $10$ & $(2,9,10,1,6,1,2,15,9,11,10,10,4,15,6,11)$\\
\hline
$\overline{11}^0$ & $10,12$ & $13$ & $(1,9,10,11,12,1,4,6,10,9,13,13,6,13)$\\
\hline
$\overline{12}^0$ & $11,13$ & $0$ & $(0,0,3,1,0,1,4,6,12,11,11,3,4,12,6,13)$\\
\hline
$\overline{13}^0$ & $12,14$ & $13$ & $(1,9,3,1,12,4,4,3,9,13,14,15,13,13,12,14)$\\
\hline
$\overline{14}^0$ & $13,15$ & $6$ & $(2,3,6,2,0,13,6,6,0,8,8,15,14,15,6,14)$\\
\hline
$\overline{15}^0$ & $0,14$ & $4$ & $(4,0,0,1,4,4,4,6,10,1,14,9,14,9,6,7)$\\
\hline
$\overline{0}^1$ & $15$ & $10$ & $(8,9,10,1,6,4,14,6,10,8,10,15,4,15,10,14)$\\
\hline
$\overline{1}^1$ & $0,2$ & $5$ & $(1,5,0,11,12,13,2,5,0,1,2,9,13,5,12,11)$\\
\hline
$\overline{2}^1$ & $1,3$ & $7$ & $(1,3,6,7,0,7,7,7,12,1,14,3,14,12,6,11)$\\
\hline
$\overline{3}^1$ & $2,4$ & $15$ & $(2,5,10,2,5,7,7,15,10,11,11,15,4,9,15,15)$\\
\hline
$\overline{4}^1$ & $3,5$ & $14$ & $(2,3,10,2,12,4,4,5,10,13,14,3,14,12,14,14)$\\
\hline
$\overline{5}^1$ & $4,6$ & $1$ & $(1,1,6,1,5,1,4,3,10,13,11,3,4,5,6,11)$\\
\hline
$\overline{6}^1$ & $5,7$ & $1$ & $(1,1,0,1,5,1,7,5,0,8,2,15,8,12,12,7)$\\
\hline
$\overline{7}^1$ & $6,8$ & $15$ & $(2,5,6,2,6,13,14,5,12,8,8,15,13,15,15,15)$\\
\hline
$\overline{8}^1$ & $7,9$ & $3$ & $(1,3,0,3,6,7,7,3,0,1,14,9,14,9,6,13)$\\
\hline
$\overline{9}^1$ & $8,10$ & $10$ & $(8,9,10,7,0,7,4,15,10,8,10,15,4,9,10,14)$\\
\hline
$\overline{10}^1$ & $9,11$ & $10$ & $(2,9,10,1,6,1,2,15,9,11,10,10,4,15,6,11)$\\
\hline
$\overline{11}^1$ & $10,12$ & $13$ & $(1,9,10,11,12,1,4,6,10,13,11,9,13,13,6,13)$\\
\hline
$\overline{12}^1$ & $11,13$ & $0$ & $(0,0,3,1,0,1,4,6,12,11,11,3,4,12,6,13)$\\
\hline
$\overline{13}^1$ & $12,14$ & $14$ & $(8,3,0,7,0,7,14,6,12,13,8,3,14,12,14,14)$\\
\hline
$\overline{14}^1$ & $13,15$ & $9$ & $(1,9,0,2,0,1,4,15,9,9,2,15,14,9,12,13)$\\
\hline
$\overline{15}^1$ & $1,14$ & $7$ & $(1,5,6,7,0,7,7,7,10,1,14,10,13,5,6,13)$\\
\hline
\caption{Local trees $T_1(C)$ in the clusters of $F_5$}
\label{local1}
\end{longtable}
\end{center}

\begin{center}
\begin{longtable}{|c|c|c|c|}
\hline
Cluster $C$ & Generators & Center & Tree $T_2(C)$ \\
\hline
$\overline{0}^0$ & $3$ & $0$ & $(0,0,0,2,0,4,14,5,10,13,2,3,14,5,10,13)$ \\
\hline
$\overline{1}^0$ & $4,13$ & $13$ & $(4,9,0,11,5,13,7,5,8,13,11,9,13,13,10,7)$ \\
\hline
$\overline{2}^0$ & $5,15$ & $9$ & $(8,5,10,2,5,13,4,15,9,9,8,10,4,9,15,13)$\\
\hline
$\overline{3}^0$ & $0,6$ & $8$ & $(1,9,6,7,0,1,7,15,9,11,11,11,8,15,6,11)$\\
\hline
$\overline{4}^0$ & $1,7$ & $11$ & $(1,9,6,7,0,1,7,15,9,11,11,11,8,15,6,11)$ \\
\hline
$\overline{5}^0$ & $2,8$ & $12$ & $(8,9,0,2,0,7,7,15,12,8,14,9,12,12,12,14)$ \\
\hline
$\overline{6}^0$ & $3,9$ & $10$ & $(4,9,6,11,6,13,14,3,10,11,10,11,4,9,10,14)$ \\
\hline
$\overline{7}^0$ & $4,10$ & $3$ & $(1,3,10,3,0,4,7,3,0,11,11,3,4,9,10,7)$ \\
\hline
$\overline{8}^0$ & $5,11$ & $8$ & $(4,5,10,11,12,4,2,15,8,8,8,10,8,12,15,11)$ \\
\hline
$\overline{9}^0$ & $6,12$ & $1$ & $(1,1,3,1,5,1,2,6,12,11,11,3,13,5,6,13)$ \\
\hline
$\overline{10}^0$ & $7,13$ & $13$ & $(8,0,3,7,0,13,7,5,12,13,8,3,13,13,12,14)$ \\
\hline
$\overline{11}^0$ & $8,14$ & $3$ & $(2,3,3,3,0,7,2,3,0,8,14,15,8,5,15,7)$ \\
\hline
$\overline{12}^0$ & $9,15$ & $14$ & $(2,9,10,7,5,7,2,15,10,8,14,15,14,5,14,14)$ \\
\hline
$\overline{13}^0$ & $0,10$ & $2$ & $(2,5,2,11,0,7,2,6,0,11,8,5,8,5,6,7)$ \\
\hline
$\overline{14}^0$ & $1,11$ & $1$ & $(1,1,10,7,5,1,4,5,12,1,11,9,4,12,20,7)$ \\
\hline
$\overline{15}^0$ & $2,12$ & $15$ & $(8,5,3,11,12,13,2,3,12,8,11,15,13,15,15,15)$ \\
\hline
$\overline{0}^1$ & $3,13$ & $7$ & $(2,0,3,7,0,7,2,3,12,11,11,3,13,5,12,7)$ \\
\hline
$\overline{1}^1$ & $4,14$ & $14$ & $(4,3,3,7,6,4,14,6,10,8,14,10,8,15,14,14)$ \\
\hline
$\overline{2}^1$ & $5,15$ & $9$ & $(8,5,10,2,5,13,4,15,9,9,8,10,4,9,15,13)$ \\
\hline
$\overline{3}^1$ & $0,6$ & $8$ & $(8,0,6,1,0,13,14,3,8,1,14,3,8,12,12,13)$ \\
\hline
$\overline{4}^1$ & $1,7$ & $11$ & $(1,9,6,7,0,1,7,15,9,11,11,11,8,15,6,11)$ \\
\hline
$\overline{5}^1$ & $2,8$ & $12$ & $(8,9,0,2,0,7,7,15,12,9,14,9,12,12,12,14)$ \\
\hline
$\overline{6}^1$ & $3,9$ & $10$ & $(4,9,6,7,6,13,14,3,10,11,10,10,4,9,10,14)$ \\
\hline
$\overline{7}^1$ & $4,10$ & $3$ & $(1,3,10,3,0,4,7,3,0,11,11,3,4,9,10,7)$ \\
\hline
$\overline{8}^1$ & $5,11$ & $8$ & $(4,5,10,11,12,4,2,15,8,8,8,10,8,12,15,11)$ \\
\hline
$\overline{9}^1$ & $6,12$ & $1$ & $(1,1,3,1,5,1,2,6,12,11,11,3,13,5,6,13)$ \\
\hline
$\overline{10}^1$ & $7,13$ & $13$ & $(8,0,3,7,0,13,7,5,12,13,8,3,13,13,12,14)$ \\
\hline
$\overline{11}^1$ & $8,14$ & $3$ & $(2,3,3,3,0,7,2,3,0,8,14,15,8,5,15,7)$ \\
\hline
$\overline{12}^1$ & $9,15$ & $14$ & $(2,9,10,7,5,7,2,15,10,8,14,15,14,5,14,14)$ \\
\hline
$\overline{13}^1$ & $1,10$ & $9$ & $(1,9,10,2,5,1,2,15,9,9,11,9,4,15,10,11)$ \\
\hline
$\overline{14}^1$ & $11$ & $14$ & $(8,3,6,7,5,7,14,6,10,11,14,10,8,5,14,14)$ \\
\hline
$\overline{15}^1$ & $2,12$ & $9$ & $(8,3,3,11,12,4,4,15,9,9,2,9,8,9,12,11)$ \\
\hline
\caption{Local trees $T_2(C)$ in the clusters of $F_5$}
\label{local2}
\end{longtable}
\end{center}

\section{Respective representation of $T_1$ and $T_2$ in $G_1$ and $G_2$} \label{appC}
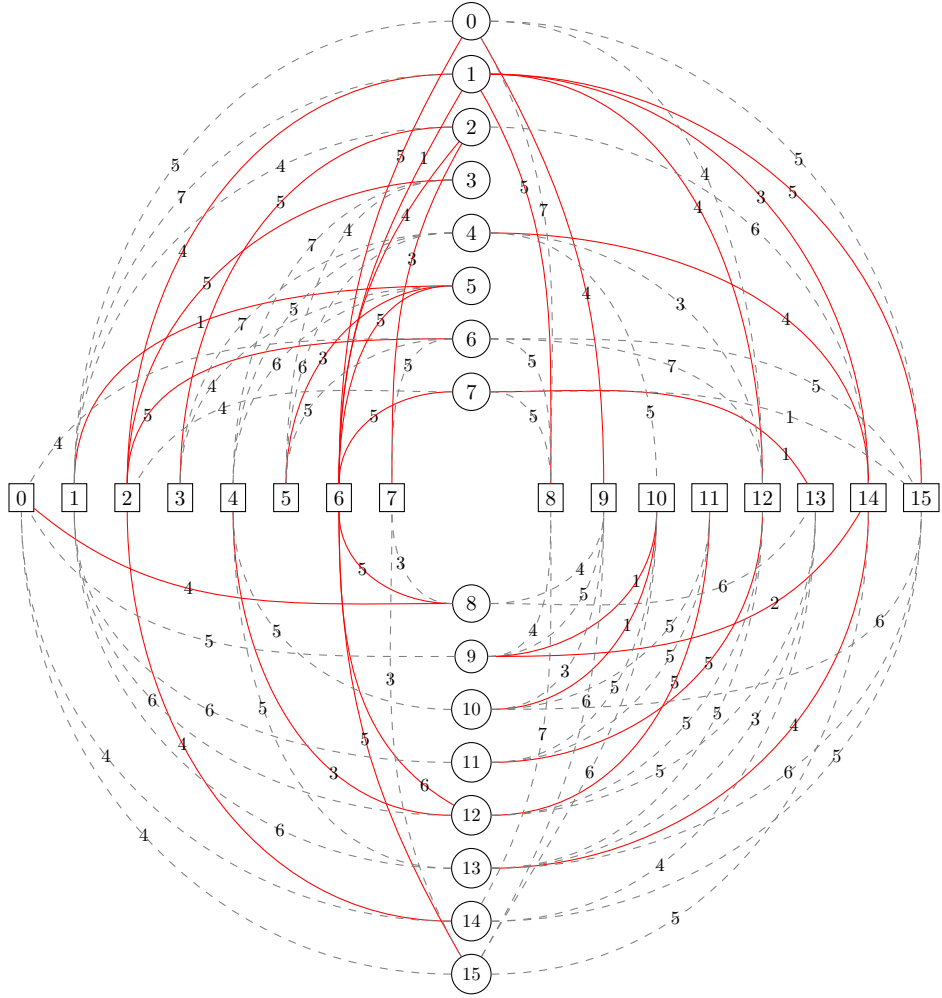
\begin{figure}[h]
    \centering
\begin{tikzpicture}[scale=0.7]
\node[shape=circle,draw=black,scale=0.8] (01) at (0,9) {$0$};
\node[shape=circle,draw=black,scale=0.8] (02) at (0,8) {$1$};
\node[shape=circle,draw=black,scale=0.8] (03) at (0,7) {$2$};
\node[shape=circle,draw=black,scale=0.8] (04) at (0,6) {$3$};
\node[shape=circle,draw=black,scale=0.8] (05) at (0,5) {$4$};
\node[shape=circle,draw=black,scale=0.8] (06) at (0,4) {$5$};
\node[shape=circle,draw=black,scale=0.8] (07) at (0,3) {$6$};
\node[shape=circle,draw=black,scale=0.8] (08) at (0,2) {$7$};
\node[shape=circle,draw=black,scale=0.8] (09) at (0,-2) {$8$};
\node[shape=circle,draw=black,scale=0.7] (010) at (0,-3) {$9$};
\node[shape=circle,draw=black,scale=0.7] (011) at (0,-4) {$10$};
\node[shape=circle,draw=black,scale=0.7] (012) at (0,-5) {$11$};
\node[shape=circle,draw=black,scale=0.7] (013) at (0,-6) {$12$};
\node[shape=circle,draw=black,scale=0.7] (014) at (0,-7) {$13$};
\node[shape=circle,draw=black,scale=0.7] (015) at (0,-8) {$14$};
\node[shape=circle,draw=black,scale=0.7] (016) at (0,-9) {$15$};

\node[shape=rectangle,draw=black,scale=0.8] (101) at (-8.5,0) {$0$};
\node[shape=rectangle,draw=black,scale=0.8] (102) at (-7.5,0) {$1$};
\node[shape=rectangle,draw=black,scale=0.8] (103) at (-6.5,0) {$2$};
\node[shape=rectangle,draw=black,scale=0.8] (104) at (-5.5,0) {$3$};
\node[shape=rectangle,draw=black,scale=0.8] (105) at (-4.5,0) {$4$};
\node[shape=rectangle,draw=black,scale=0.8] (106) at (-3.5,0) {$5$};
\node[shape=rectangle,draw=black,scale=0.8] (107) at (-2.5,0) {$6$};
\node[shape=rectangle,draw=black,scale=0.8] (108) at (-1.5,0) {$7$};
\node[shape=rectangle,draw=black,scale=0.8] (109) at (1.5,0) {$8$};
\node[shape=rectangle,draw=black,scale=0.8] (110) at (2.5,0) {$9$};
\node[shape=rectangle,draw=black,scale=0.8] (111) at (3.5,0) {$10$};
\node[shape=rectangle,draw=black,scale=0.8] (112) at (4.5,0) {$11$};
\node[shape=rectangle,draw=black,scale=0.8] (113) at (5.5,0) {$12$};
\node[shape=rectangle,draw=black,scale=0.8] (114) at (6.5,0) {$13$};
\node[shape=rectangle,draw=black,scale=0.8] (115) at (7.5,0) {$14$};
\node[shape=rectangle,draw=black,scale=0.8] (116) at (8.5,0) {$15$};

    \draw[red] (01) to[out=-120, in= 90] node[pos=0.27, black,scale=0.7] {$5$} (107);
    \draw[red] (01) to[out=-60, in= 90] node[pos=0.6, black,scale=0.7] {$4$} (110);
    \draw[red] (02) to[out=180, in= 90] node[pos=0.6, black,scale=0.7] {$4$} (103);
    \draw[red] (02) to[out=-120, in= 90] node[pos=0.17, black,scale=0.7] {$1$} (107);
    \draw[red] (02) to[out=-60, in= 90] node[pos=0.25, black,scale=0.7] {$5$} (109);
    \draw[red] (02) to[out=0, in= 90] node[pos=0.5, black,scale=0.7] {$4$} (113);
    \draw[red] (02) to[out=0, in= 90] node[pos=0.5, black,scale=0.7] {$3$} (115);
    \draw[red] (02) to[out=0, in= 90] node[pos=0.5, black,scale=0.7] {$5$} (116);
    \draw[red] (03) to[out=180, in= 90] node[pos=0.4, black,scale=0.7] {$5$} (104);
    \draw[red] (03) to[out=-130, in= 90] node[pos=0.23, black,scale=0.7] {$4$} (107);
    \draw[red] (03) to[out=-120, in= 90] node[pos=0.35, black,scale=0.7] {$3$} (108);
    \draw[red] (04) to[out=180, in= 90] node[pos=0.55, black,scale=0.7] {$5$} (103);
    \draw[red] (05) to[out=0, in= 90] node[pos=0.6, black,scale=0.7] {$4$} (115);
    \draw[red] (06) to[out=180, in= 90] node[pos=0.5, black,scale=0.7] {$1$} (102);
    \draw[red] (06) to[out=180, in= 90] node[pos=0.55, black,scale=0.7] {$3$} (106);
    \draw[red] (06) to[out=180, in= 90] node[pos=0.35, black,scale=0.7] {$5$} (107);
    \draw[red] (07) to[out=180, in= 90] node[pos=0.8, black,scale=0.7] {$5$} (103);
    \draw[red] (08) to[out=180, in= 90] node[pos=0.5, black,scale=0.7] {$5$} (107);
    \draw[red] (08) to[out=0, in= 120] node[pos=0.9, black,scale=0.7] {$1$} (114);
    \draw[red] (09) to[out=180, in= -40] node[pos=0.6, black,scale=0.7] {$4$} (101);
    \draw[red] (09) to[out=180, in= -90] node[pos=0.6, black,scale=0.7] {$5$} (107);
    \draw[red] (010) to[out=0, in= -90] node[pos=0.7, black,scale=0.7] {$1$} (111);
    \draw[red] (010) to[out=0, in= -120] node[pos=0.7, black,scale=0.7] {$2$} (115);
    \draw[red] (011) to[out=0, in= -90] node[pos=0.6, black,scale=0.7] {$1$} (111);
    \draw[red] (012) to[out=0, in= -90] node[pos=0.6, black,scale=0.7] {$5$} (113);
    \draw[red] (013) to[out=180, in= -90] node[pos=0.32, black,scale=0.7] {$3$} (105);
    \draw[red] (013) to[out=150, in= -90] node[pos=0.1, black,scale=0.7] {$6$} (107);
    \draw[red] (013) to[out=0, in= -90] node[pos=0.6, black,scale=0.7] {$5$} (112);
    \draw[red] (014) to[out=0, in= -90] node[pos=0.6, black,scale=0.7] {$4$} (115);
    \draw[red] (015) to[out=180, in= -90] node[pos=0.6, black,scale=0.7] {$4$} (103);
    \draw[red] (016) to[out=120, in= -90] node[pos=0.51, black,scale=0.7] {$5$} (107);

\draw[dashed,gray] (01) to[out=-60, in= 90] node[pos=0.4, black,scale=0.7] {$7$} (109);
\draw[dashed,gray] (01) to[out=180, in= 90] node[pos=0.5, black,scale=0.7] {$5$} (102);
\draw[dashed,gray] (01) to[out=0, in= 90] node[pos=0.5, black,scale=0.7] {$4$} (113);
\draw[dashed,gray] (01) to[out=0, in= 90] node[pos=0.5, black,scale=0.7] {$5$} (116);

\draw[dashed,gray] (02) to[out=180, in= 90] node[pos=0.5, black,scale=0.7] {$7$} (102);

\draw[dashed,gray] (03) to[out=180, in= 90] node[pos=0.3, black,scale=0.7] {$4$} (102);
\draw[dashed,gray] (03) to[out=0, in= 90] node[pos=0.5, black,scale=0.7] {$6$} (115);

\draw[dashed,gray] (04) to[out=180, in= 90] node[pos=0.4, black,scale=0.7] {$7$} (105);
\draw[dashed,gray] (04) to[out=180, in= 90] node[pos=0.34, black,scale=0.7] {$4$} (106);

\draw[dashed,gray] (05) to[out=180, in= 90] node[pos=0.57, black,scale=0.7] {$7$} (104);
\draw[dashed,gray] (05) to[out=180, in= 90] node[pos=0.5, black,scale=0.7] {$5$} (105);
\draw[dashed,gray] (05) to[out=180, in= 90] node[pos=0.67, black,scale=0.7] {$6$} (106);
\draw[dashed,gray] (05) to[out=0, in= 90] node[pos=0.8, black,scale=0.7] {$5$} (111);
\draw[dashed,gray] (05) to[out=0, in= 90] node[pos=0.5, black,scale=0.7] {$3$} (113);

\draw[dashed,gray] (06) to[out=180, in= 90] node[pos=0.71, black,scale=0.7] {$4$} (104);
\draw[dashed,gray] (06) to[out=180, in= 90] node[pos=0.6, black,scale=0.7] {$6$} (105);

\draw[dashed,gray] (07) to[out=180, in= 60] node[pos=0.9, black,scale=0.7] {$4$} (101);
\draw[dashed,gray] (07) to[out=180, in= 90] node[pos=0.68, black,scale=0.7] {$5$} (106);
\draw[dashed,gray] (07) to[out=180, in= 90] node[pos=0.35, black,scale=0.7] {$5$} (108);
\draw[dashed,gray] (07) to[out=0, in= 90] node[pos=0.32, black,scale=0.7] {$5$} (109);
\draw[dashed,gray] (07) to[out=0, in= 90] node[pos=0.5, black,scale=0.7] {$7$} (113);
\draw[dashed,gray] (07) to[out=0, in= 120] node[pos=0.7, black,scale=0.7] {$5$} (116);

\draw[dashed,gray] (08) to[out=180, in= 60] node[pos=0.65, black,scale=0.7] {$4$} (103);
\draw[dashed,gray] (08) to[out=0, in= 90] node[pos=0.45, black,scale=0.7] {$5$} (109);
\draw[dashed,gray] (08) to[out=0, in= 140] node[pos=0.7, black,scale=0.7] {$1$} (116);

\draw[dashed,gray] (09) to[out=180, in= -90] node[pos=0.6, black,scale=0.7] {$3$} (108);
\draw[dashed,gray] (09) to[out=0, in= -90] node[pos=0.6, black,scale=0.7] {$4$} (110);
\draw[dashed,gray] (09) to[out=0, in= -120] node[pos=0.66, black,scale=0.7] {$6$} (114);

\draw[dashed,gray] (010) to[out=180, in= -60] node[pos=0.5, black,scale=0.7] {$5$} (101);
\draw[dashed,gray] (010) to[out=0, in= -90] node[pos=0.35, black,scale=0.7] {$4$} (109);
\draw[dashed,gray] (010) to[out=0, in= -90] node[pos=0.6, black,scale=0.7] {$5$} (110);

\draw[dashed,gray] (011) to[out=180, in= -90] node[pos=0.6, black,scale=0.7] {$5$} (105);
\draw[dashed,gray] (011) to[out=0, in= -90] node[pos=0.37, black,scale=0.7] {$3$} (110);
\draw[dashed,gray] (011) to[out=0, in= -90] node[pos=0.62, black,scale=0.7] {$5$} (112);
\draw[dashed,gray] (011) to[out=0, in= -90] node[pos=0.75, black,scale=0.7] {$6$} (116);

\draw[dashed,gray] (012) to[out=180, in= -90] node[pos=0.47, black,scale=0.7] {$6$} (102);
\draw[dashed,gray] (012) to[out=0, in= -90] node[pos=0.48, black,scale=0.7] {$5$} (111);
\draw[dashed,gray] (012) to[out=0, in= -90] node[pos=0.6, black,scale=0.7] {$5$} (112);

\draw[dashed,gray] (013) to[out=180, in= -90] node[pos=0.6, black,scale=0.7] {$6$} (102);
\draw[dashed,gray] (013) to[out=0, in= -90] node[pos=0.5, black,scale=0.7] {$5$} (113);
\draw[dashed,gray] (013) to[out=0, in= -90] node[pos=0.35, black,scale=0.7] {$5$} (114);

\draw[dashed,gray] (014) to[out=-180, in= -90] node[pos=0.3, black,scale=0.7] {$6$} (102);
\draw[dashed,gray] (014) to[out=-180, in= -90] node[pos=0.6, black,scale=0.7] {$5$} (105);
\draw[dashed,gray] (014) to[out=0, in= -90] node[pos=0.6, black,scale=0.7] {$5$} (113);
\draw[dashed,gray] (014) to[out=0, in= -90] node[pos=0.6, black,scale=0.7] {$3$} (114);
\draw[dashed,gray] (014) to[out=0, in= -90] node[pos=0.5, black,scale=0.7] {$6$} (116);

\draw[dashed,gray] (015) to[out=180, in= -90] node[pos=0.6, black,scale=0.7] {$4$} (101);
\draw[dashed,gray] (015) to[out=60, in= -90] node[pos=0.45, black,scale=0.7] {$7$} (109);
\draw[dashed,gray] (015) to[out=0, in= -90] node[pos=0.32, black,scale=0.7] {$4$} (114);
\draw[dashed,gray] (015) to[out=0, in= -90] node[pos=0.6, black,scale=0.7] {$5$} (116);

\draw[dashed,gray] (016) to[out=180, in= -90] node[pos=0.5, black,scale=0.7] {$4$} (101);
\draw[dashed,gray] (016) to[out=120, in= -90] node[pos=0.65, black,scale=0.7] {$3$} (108);
\draw[dashed,gray] (016) to[out=60, in= -90] node[pos=0.6, black,scale=0.7] {$6$} (110);
\draw[dashed,gray] (016) to[out=60, in= -90] node[pos=0.43, black,scale=0.7] {$6$} (111);
\draw[dashed,gray] (016) to[out=0, in= -90] node[pos=0.3, black,scale=0.7] {$5$} (115);
\end{tikzpicture}
    \caption{The weighted graph $G_1$ and the tree $T_1$ (plain red edges) that minimizes the diameter (eccentricity from cluster $\overline{3}^0$ to cluster $\overline{1}^1$ through the center $\overline{6}^1$)}
    \label{g1}
\end{figure}
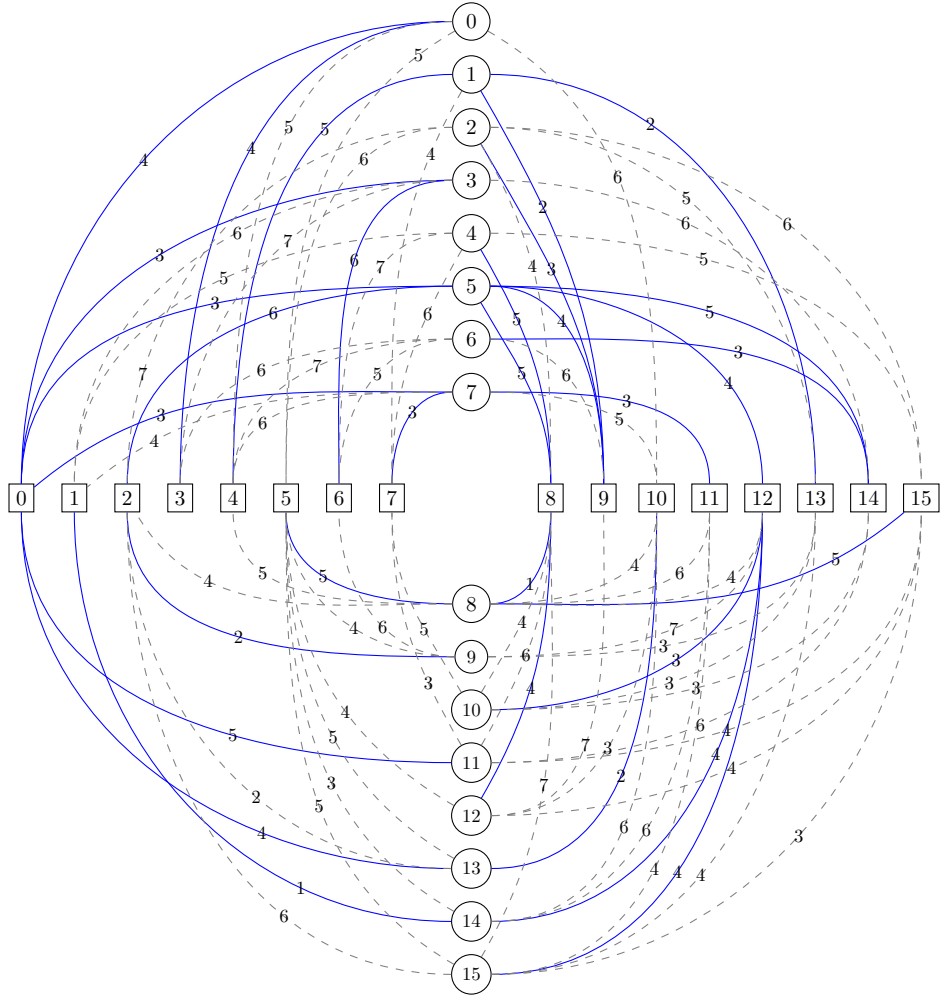
\begin{figure}[h]
    \centering
\begin{tikzpicture}[scale=0.7]
\node[shape=circle,draw=black,scale=0.8] (01) at (0,9) {$0$};
\node[shape=circle,draw=black,scale=0.8] (02) at (0,8) {$1$};
\node[shape=circle,draw=black,scale=0.8] (03) at (0,7) {$2$};
\node[shape=circle,draw=black,scale=0.8] (04) at (0,6) {$3$};
\node[shape=circle,draw=black,scale=0.8] (05) at (0,5) {$4$};
\node[shape=circle,draw=black,scale=0.8] (06) at (0,4) {$5$};
\node[shape=circle,draw=black,scale=0.8] (07) at (0,3) {$6$};
\node[shape=circle,draw=black,scale=0.8] (08) at (0,2) {$7$};
\node[shape=circle,draw=black,scale=0.8] (09) at (0,-2) {$8$};
\node[shape=circle,draw=black,scale=0.7] (010) at (0,-3) {$9$};
\node[shape=circle,draw=black,scale=0.7] (011) at (0,-4) {$10$};
\node[shape=circle,draw=black,scale=0.7] (012) at (0,-5) {$11$};
\node[shape=circle,draw=black,scale=0.7] (013) at (0,-6) {$12$};
\node[shape=circle,draw=black,scale=0.7] (014) at (0,-7) {$13$};
\node[shape=circle,draw=black,scale=0.7] (015) at (0,-8) {$14$};
\node[shape=circle,draw=black,scale=0.7] (016) at (0,-9) {$15$};

\node[shape=rectangle,draw=black,scale=0.8] (101) at (-8.5,0) {$0$};
\node[shape=rectangle,draw=black,scale=0.8] (102) at (-7.5,0) {$1$};
\node[shape=rectangle,draw=black,scale=0.8] (103) at (-6.5,0) {$2$};
\node[shape=rectangle,draw=black,scale=0.8] (104) at (-5.5,0) {$3$};
\node[shape=rectangle,draw=black,scale=0.8] (105) at (-4.5,0) {$4$};
\node[shape=rectangle,draw=black,scale=0.8] (106) at (-3.5,0) {$5$};
\node[shape=rectangle,draw=black,scale=0.8] (107) at (-2.5,0) {$6$};
\node[shape=rectangle,draw=black,scale=0.8] (108) at (-1.5,0) {$7$};
\node[shape=rectangle,draw=black,scale=0.8] (109) at (1.5,0) {$8$};
\node[shape=rectangle,draw=black,scale=0.8] (110) at (2.5,0) {$9$};
\node[shape=rectangle,draw=black,scale=0.8] (111) at (3.5,0) {$10$};
\node[shape=rectangle,draw=black,scale=0.8] (112) at (4.5,0) {$11$};
\node[shape=rectangle,draw=black,scale=0.8] (113) at (5.5,0) {$12$};
\node[shape=rectangle,draw=black,scale=0.8] (114) at (6.5,0) {$13$};
\node[shape=rectangle,draw=black,scale=0.8] (115) at (7.5,0) {$14$};
\node[shape=rectangle,draw=black,scale=0.8] (116) at (8.5,0) {$15$};

    \draw[blue] (01) to[out=180, in= 90] node[pos=0.5, black,scale=0.7] {$4$} (101);
    \draw[blue] (01) to[out=180, in= 90] node[pos=0.45, black,scale=0.7] {$4$} (104);
    \draw[blue] (02) to[out=180, in= 90] node[pos=0.3, black,scale=0.7] {$5$} (105);
    \draw[blue] (02) to[out=-60, in= 90] node[pos=0.3, black,scale=0.7] {$2$} (110);
    \draw[blue] (02) to[out=0, in= 90] node[pos=0.3, black,scale=0.7] {$2$} (114);
    \draw[blue] (03) to[out=-60, in= 90] node[pos=0.38, black,scale=0.7] {$3$} (110);
    \draw[blue] (04) to[out=180, in= 90] node[pos=0.5, black,scale=0.7] {$3$} (101);
    \draw[blue] (04) to[out=180, in= 90] node[pos=0.43, black,scale=0.7] {$6$} (107);
    \draw[blue] (05) to[out=-60, in= 90] node[pos=0.3, black,scale=0.7] {$5$} (109);
    \draw[blue] (06) to[out=180, in= 90] node[pos=0.41, black,scale=0.7] {$3$} (101);
    \draw[blue] (06) to[out=180, in= 90] node[pos=0.4, black,scale=0.7] {$6$} (103);
    \draw[blue] (06) to[out=-60, in= 90] node[pos=0.4, black,scale=0.7] {$5$} (109);
    \draw[blue] (06) to[out=0, in= 90] node[pos=0.35, black,scale=0.7] {$4$} (110);
    \draw[blue] (06) to[out=0, in= 90] node[pos=0.7, black,scale=0.7] {$4$} (113);
    \draw[blue] (06) to[out=0, in= 90] node[pos=0.43, black,scale=0.7] {$5$} (115);
    \draw[blue] (07) to[out=0, in= 90] node[pos=0.5, black,scale=0.7] {$3$} (115);
    \draw[blue] (08) to[out=180, in= 40] node[pos=0.67, black,scale=0.7] {$3$} (101);
    \draw[blue] (08) to[out=180, in= 90] node[pos=0.4, black,scale=0.7] {$3$} (108);
    \draw[blue] (08) to[out=0, in= 90] node[pos=0.47, black,scale=0.7] {$3$} (112);
    \draw[blue] (09) to[out=180, in= -90] node[pos=0.6, black,scale=0.7] {$5$} (106);
    \draw[blue] (09) to[out=0, in= -90] node[pos=0.4, black,scale=0.7] {$1$} (109);
    \draw[blue] (09) to[out=0, in= -140] node[pos=0.82, black,scale=0.7] {$5$} (116);
    \draw[blue] (010) to[out=180, in= -90] node[pos=0.5, black,scale=0.7] {$2$} (103);
    \draw[blue] (011) to[out=0, in= -90] node[pos=0.5, black,scale=0.7] {$3$} (113);
    \draw[blue] (012) to[out=180, in= -90] node[pos=0.37, black,scale=0.7] {$5$} (101);
    \draw[blue] (013) to[out=60, in= -90] node[pos=0.4, black,scale=0.7] {$4$} (109);
    \draw[blue] (014) to[out=180, in= -90] node[pos=0.3, black,scale=0.7] {$4$} (101);
    \draw[blue] (014) to[out=0, in= -90] node[pos=0.43, black,scale=0.7] {$2$} (111);
    \draw[blue] (015) to[out=180, in= -90] node[pos=0.25, black,scale=0.7] {$1$} (102);    
    \draw[blue] (015) to[out=0, in= -90] node[pos=0.57, black,scale=0.7] {$4$} (113);
    \draw[blue] (016) to[out=0, in= -90] node[pos=0.4, black,scale=0.7] {$4$} (113);

\draw[gray,dashed] (01) to[out=180, in= 90] node[pos=0.4, black,scale=0.7] {$5$} (105);
\draw[gray,dashed] (01) to[out=-150, in= 90] node[pos=0.08, black,scale=0.7] {$5$} (106);
\draw[gray,dashed] (01) to[out=-30, in= 90] node[pos=0.4, black,scale=0.7] {$6$} (111);

\draw[gray,dashed] (02) to[out=-120, in= 90] node[pos=0.16, black,scale=0.7] {$4$} (108);

\draw[gray,dashed] (03) to[out=180, in= 90] node[pos=0.8, black,scale=0.7] {$7$} (103);
\draw[gray,dashed] (03) to[out=180, in= 90] node[pos=0.24, black,scale=0.7] {$6$} (106);
\draw[gray,dashed] (03) to[out=-60, in= 90] node[pos=0.37, black,scale=0.7] {$4$} (109);
\draw[gray,dashed] (03) to[out=0, in= 90] node[pos=0.4, black,scale=0.7] {$5$} (114);
\draw[gray,dashed] (03) to[out=0, in= 90] node[pos=0.5, black,scale=0.7] {$6$} (116);

\draw[gray,dashed] (04) to[out=180, in= 90] node[pos=0.4, black,scale=0.7] {$6$} (102);
\draw[gray,dashed] (04) to[out=180, in= 90] node[pos=0.4, black,scale=0.7] {$7$} (104);
\draw[gray,dashed] (04) to[out=0, in= 90] node[pos=0.36, black,scale=0.7] {$6$} (115);

\draw[gray,dashed] (05) to[out=180, in= 90] node[pos=0.44, black,scale=0.7] {$5$} (102);
\draw[gray,dashed] (05) to[out=180, in= 90] node[pos=0.3, black,scale=0.7] {$7$} (107);
\draw[gray,dashed] (05) to[out=-120, in= 90] node[pos=0.28, black,scale=0.7] {$6$} (108);
\draw[gray,dashed] (05) to[out=0, in= 90] node[pos=0.36, black,scale=0.7] {$5$} (116);

\draw[gray,dashed] (07) to[out=180, in= 90] node[pos=0.53, black,scale=0.7] {$6$} (104);
\draw[gray,dashed] (07) to[out=180, in= 90] node[pos=0.45, black,scale=0.7] {$7$} (105);
\draw[gray,dashed] (07) to[out=180, in= 90] node[pos=0.43, black,scale=0.7] {$5$} (107);
\draw[gray,dashed] (07) to[out=0, in= 90] node[pos=0.44, black,scale=0.7] {$6$} (110);

\draw[gray,dashed] (08) to[out=180, in= 40] node[pos=0.8, black,scale=0.7] {$4$} (102);
\draw[gray,dashed] (08) to[out=180, in= 90] node[pos=0.7, black,scale=0.7] {$6$} (105);
\draw[gray,dashed] (08) to[out=0, in= 90] node[pos=0.6, black,scale=0.7] {$5$} (111);

\draw[gray,dashed] (09) to[out=180, in= -60] node[pos=0.7, black,scale=0.7] {$4$} (103);
\draw[gray,dashed] (09) to[out=180, in= -90] node[pos=0.7, black,scale=0.7] {$5$} (105);
\draw[gray,dashed] (09) to[out=0, in= -90] node[pos=0.7, black,scale=0.7] {$4$} (111);
\draw[gray,dashed] (09) to[out=0, in= -90] node[pos=0.7, black,scale=0.7] {$6$} (112);
\draw[gray,dashed] (09) to[out=0, in= -90] node[pos=0.73, black,scale=0.7] {$4$} (113);

\draw[gray,dashed] (010) to[out=180, in= -90] node[pos=0.42, black,scale=0.7] {$4$} (106);
\draw[gray,dashed] (010) to[out=180, in= -90] node[pos=0.4, black,scale=0.7] {$6$} (107);
\draw[gray,dashed] (010) to[out=0, in= -90] node[pos=0.51, black,scale=0.7] {$7$} (113);
\draw[gray,dashed] (010) to[out=0, in= -90] node[pos=0.4, black,scale=0.7] {$3$} (114);

\draw[gray,dashed] (011) to[out=120, in= -90] node[pos=0.36, black,scale=0.7] {$5$} (108);
\draw[gray,dashed] (011) to[out=60, in= -90] node[pos=0.4, black,scale=0.7] {$4$} (109);
\draw[gray,dashed] (011) to[out=0, in= -90] node[pos=0.4, black,scale=0.7] {$3$} (114);
\draw[gray,dashed] (011) to[out=0, in= -90] node[pos=0.4, black,scale=0.7] {$3$} (115);

\draw[gray,dashed] (012) to[out=120, in= -90] node[pos=0.27, black,scale=0.7] {$3$} (108);
\draw[gray,dashed] (012) to[out=60, in= -90] node[pos=0.4, black,scale=0.7] {$6$} (109);
\draw[gray,dashed] (012) to[out=0, in= -90] node[pos=0.4, black,scale=0.7] {$6$} (115);
\draw[gray,dashed] (012) to[out=0, in= -90] node[pos=0.4, black,scale=0.7] {$4$} (116);

\draw[gray,dashed] (013) to[out=150, in= -90] node[pos=0.4, black,scale=0.7] {$4$} (106);
\draw[gray,dashed] (013) to[out=0, in= -90] node[pos=0.4, black,scale=0.7] {$7$} (110);
\draw[gray,dashed] (013) to[out=0, in= -90] node[pos=0.4, black,scale=0.7] {$3$} (111);
\draw[gray,dashed] (013) to[out=0, in= -90] node[pos=0.4, black,scale=0.7] {$4$} (116);

\draw[gray,dashed] (014) to[out=180, in= -90] node[pos=0.4, black,scale=0.7] {$2$} (103);
\draw[gray,dashed] (014) to[out=150, in= -90] node[pos=0.43, black,scale=0.7] {$5$} (106);

\draw[gray,dashed] (015) to[out=150, in= -90] node[pos=0.4, black,scale=0.7] {$3$} (106);
\draw[gray,dashed] (015) to[out=0, in= -90] node[pos=0.4, black,scale=0.7] {$6$} (111);
\draw[gray,dashed] (015) to[out=0, in= -90] node[pos=0.4, black,scale=0.7] {$6$} (112);

\draw[gray,dashed] (016) to[out=180, in= -90] node[pos=0.3, black,scale=0.7] {$6$} (103);
\draw[gray,dashed] (016) to[out=150, in= -90] node[pos=0.43, black,scale=0.7] {$5$} (106);
\draw[gray,dashed] (016) to[out=60, in= -90] node[pos=0.4, black,scale=0.7] {$7$} (109);
\draw[gray,dashed] (016) to[out=0, in= -90] node[pos=0.4, black,scale=0.7] {$4$} (112);
\draw[gray,dashed] (016) to[out=0, in= -90] node[pos=0.4, black,scale=0.7] {$4$} (114);
\draw[gray,dashed] (016) to[out=0, in= -90] node[pos=0.5, black,scale=0.7] {$3$} (116);
\end{tikzpicture}
    \caption{The weighted graph $G_2$ and the tree $T_2$ (plain blue edges) that minimizes the diameter (eccentricity from cluster $\overline{4}^1$ to cluster $\overline{6}^1$ through the center $\overline{5}^0$)}
    \label{g2}
\end{figure}
 
\end{document}